\documentclass[a4paper]{amsproc}
\usepackage{amssymb}
\usepackage{amscd}
\usepackage[dvips]{graphicx}
\usepackage[all,cmtip]{xy}
\usepackage[inline]{enumitem}
\usepackage{xcolor}
\usepackage{cite}
\usepackage[hidelinks]{hyperref}

\makeatletter
\renewcommand*{\eqref}[1]{%
  \hyperref[{#1}]{\textup{\tagform@{\ref*{#1}}}}%
}
\makeatother

\theoremstyle{plain}

 \newtheorem{thm}{Theorem}
 \newtheorem{prop}{Proposition}
 \newtheorem{lem}{Lemma}
 \newtheorem{cor}{Corollary}

\theoremstyle{definition}

 \newtheorem{exm}{Example}
 
\newtheorem{dfn}{Definition}

\numberwithin{equation}{section}

\allowdisplaybreaks

\def\ji {\char'032}
\def\ja {\char'037}
\def\m  {\char'176}

\font\srit=wncyi8
 \font\srrm=wncyr8

\newcommand{\R}{\mathbb{R}}

\newcommand{\T}{\mathbb{T}}

\DeclareMathOperator{\Span}{\mathrm{span\,}}

\renewcommand{\le}{\leqslant}

\renewcommand{\setminus}{\smallsetminus}

\title[Contact line bundles,  foliations, and integrability]{CONTACT LINE BUNDLES,  FOLIATIONS, AND INTEGRABILITY}
\author[Jovanovi\'c]{\bfseries Bo\v zidar Jovanovi\'c}

\address{
Mathematical Institute SANU \\
Serbian Academy of Sciences and Arts \\
Kneza Mihaila 36, 11000 Belgrade\\
Serbia}
\email{bozaj@mi.sanu.ac.rs}

\begin{document}

\begin{abstract}
We formulate the non-commutative integrability of contact systems on a contact manifold $(M,\mathcal H)$ using the Jacobi structure on the space of sections $\Gamma(L)$ of a contact line bundle $L$. In the cooriented case, if the line bundle is trivial and $\mathcal H$ is the kernel of a globally defined contact form $\alpha$, the Jacobi structure on the space of sections reduces to the standard Jacobi structure on $(M,\alpha)$.
We therefore treat contact systems on cooriented and non-cooriented contact manifolds simultaneously. In particular, this allows us to work with dissipative Hamiltonian systems where the Hamiltonian does not have to be preserved by the Reeb vector field.
\end{abstract}

\maketitle

\section{Introduction}

A \emph{contact form} $\alpha$ on a $(2n+1)$-dimensional manifold
$M$ is a 1--form that satisfies the condition $\alpha\wedge(d\alpha)^n\ne0$.
By a \emph{contact manifold} $(M,\mathcal H)$ we mean a connected
$(2n+1)$-dimensional manifold $M$ equipped with a non-integrable
{\it contact} (or \emph{horizontal})
\emph{distribution} $\mathcal H$ such that for any point
$x$ there exists an open neighborhood $U$ and a contact form $\alpha_U$ on $U$ that defines the distribution: $\mathcal
H\vert_U=\ker\alpha_U$.
A vector field $X$ is an \emph{infinitesimal automorphism} or a \emph{contact vector field}
if the corresponding local 1-parameter group $\phi^t$ consists of contact diffeomorphisms ($\phi^t_*\mathcal H=\mathcal H$).
With $\mathfrak X(M,\mathcal H)$ we denote the algebra of contact vector fields.
We refer to
\begin{equation}\label{eq2}
\dot x=X, \qquad X\in \mathfrak X(M,\mathcal H)
\end{equation}
as a \emph{contact system} on $(M,\mathcal H)$.
The concept of complete integrability in symplectic geometry, namely the Arnold--Liouville integrability \cite{Ar} and the non-commutative integrability introduced by Nekhoroshev \cite{N} and Mischenko and Fomenko \cite{MF, TF}, is introduced and studied for dynamical systems defined on manifolds with other geometric structures, such as Poisson manifolds, almost-symplectic, $b$-symplectic manifolds, Dirac manifolds, co-symplectic manifolds
(see e.g. \cite{LMV, BM, KT, Jo, FS, KM, JL, Zung}), and, in a particularly important class of contact manifolds.

If $\mathcal H$ is globally defined by a contact form $\alpha$, we say that $(M,\alpha)$ is a \emph{cooriented contact manifold}.
On $(M,\alpha)$ we can assign a contact Hamiltonian vector field $X_f$ to each smooth function $f\in\mathcal A(M)$ such that $f=\alpha(X_f)$.
The space of smooth functions $\mathcal A(M)$ on $(M,\alpha)$ is a Lie algebra with respect to the Jacobi brackets, which does not satisfy the Leibnitz rule (e.g. see Chapter V in \cite{LM}, we recall the basic definitions in Section \ref{sec2}).
The complete integrability of contact Hamiltonian vector fields using the Jacobi bracket, analogous to the Arnold--Liouville integrability \cite{Ar} with respect to the Poisson bracket, is introduced in \cite{BM}, while the non-commutative contact integrability, analogous to the Nekhoroshev-Mischenko-Fomenko integrability \cite{N, MF, BJ}, is given in \cite{Jo}.
Examples of contact integrable systems can be found in  \cite{KT, Boyer, JJ1, JJ2, JJ3, SS}.

Furthermore, in \cite{Jo}, using the framework of pre-isotropic contact structures, we introduced the concept of non-commutative integrability of the contact system
\eqref{eq2}
on contact manifolds $(M,\mathcal H)$, without requiring on $\mathcal H$ to be cooriented, where invariant toric fibrations are pre-isotropic (see Definition \ref{defJov}, Section \ref{sec2}). Toric fibrations are generated by \emph{commuting contact symmetries}
$\mathcal X=\{X_0=X,X_1,\dots,X_r\}$,
\[
[X_i,X_j]=0, \qquad X_i\in \mathfrak X(M,\mathcal H), \qquad i,j=0,\dots,r.
\]

In both cases, analogous to generalized action-angle variables in the case of Hamiltonian systems integrable in a non-commutative sense \cite{N}, there exist contact action-angle variables, i.e. a normal contact form defining the contact structure $\mathcal H$ in a toroidal domain of any regular torus with quasi-periodic dynamics (Section \ref{sec2}).

On the other hand, based on the concept of so-called non-Hamiltonian toric integrability (see \cite{Bo, Koz, Zu}),
Zung in \cite{Zung} defined the contact toric integrability simply by requiring sufficient integrals and commuting contact symmetries of the system \eqref{eq2}
(see Definition \ref{defZung}, Section \ref{sec2}). He also obtained suitable contact normal forms in a toroidal domain of any regular torus, but without the pre-isotropic assumption on toric fibration.
Note that commuting symmetries $\mathcal X$ generate a pre-isotropic foliation if and only if they are transversal to the contact distribution $\mathcal H$
(Lemma \ref{lemT}, Section \ref{sec2}).

The aim of this work is to relate approaches of \cite{Jo} and \cite{Zung}.
To this end, in Section \ref{sec3}, we use the identification of symmetries with the sections of contact line bundles.
Namely, we consider the contact structure $(M,\mathcal H)$ as a line bundle $L$ over $M$ together with a contact form $\alpha\in \Omega^1(L)$. The contact vector fields $\mathfrak X(M,\mathcal H)$ are naturally identified with the space of smooth sections of the line bundle $\Gamma(L)$: for each section $s\in\Gamma(L)$ there is the unique contact vector field $X_s$ such that
\[
s=\alpha(X_s).
\]

By analogy with the cooriented case, we also call $X_s$ the \emph{contact Hamiltonian vector field with respect to the Hamiltonian} $s$.
The space of sections carries the Jacobi structure $[\cdot,\cdot]$, and a section $s$ is the symmetry of the contact system
\begin{equation}\label{eq3}
\dot x=X_h,
\end{equation}
if $[s,h]=0$. For a given set of contact symmetries $\mathbf Y=\{s_0=h,s_1,\dots,s_p\}\subset \Gamma(L)$, one associates as a standard construction in the theory of line bundles (see \cite{GH}) the momentum map $\pi\colon M\setminus M_{\mathbf 0}\to \mathbb{RP}^p$,
\[
\pi(x)=[s_0(x),\dots,s_p(x)], \qquad x\in M\setminus M_{\mathbf 0},
\]
where $M_{\mathbf 0}$ is the intersection of all zero divisors $(s_i)$ of the symmetries $\mathbf Y$.
The fibers of the momentum map as well as the zero-locus $M_{\mathbf 0}$ are invariant varieties of the contact flow of the system \eqref{eq3}
(Theorem \ref{pomocna}, Section \ref{sec3}).

Following Nekhoroshev \cite{N}, and Mischenko and Fomenko \cite{MF, TF}, we formulate contact non-commutative contact integrability in terms of the sections of the contact line bundles and the given Jacobi structure such that the fibers of the momentum map $\pi$ are generated by Hamiltonian vector fields with commuting symmetries $\mathbf X\subset \mathbf Y$. Then the compact regular components of the fibers of $\pi$ are $(r+1)$-dimensional tori with quasi-periodic dynamics ($r+1=\dim\mathbf X=2n+1-p$). Furthermore, regular compact components of the  zero-locus $M_{\mathbf 0}$ are $r$-dimensional tori with quasi-periodic dynamics, which are also generated by $\mathbf X$ (Theorem \ref{LIN}, Section \ref{sec3}). In Example \ref{Primer} (Section \ref{sec4}) and Example \ref{Primer2} (Section \ref{sec5})
we provide simple contact integrable systems on
the projective space bundle of the cotangent bundle of the torus $\mathbb T^{n+1}$
 with compact and non-compact invariant manifolds.

We prove that the Hamiltonian vector fields of symmetries $\mathbf Y$ generate an integrable distribution and describe the associate flag of the foliations $\mathcal F\subset\mathcal E$, where $\mathcal F$ is induced by $\mathbf X$ and $\mathcal E$ by $\mathbf Y$
(Theorem \ref{foliacije}, Section \ref{sec4}). In general:
\footnote{In this paper, we use the same symbol to denote the foliation and the associated integrable distribution of the tangent bundle. Furthermore, $\mathcal F_x\subset T_xM$ denotes the tangent space of the leaf through $x$.}
\[
\dim\mathcal E_x=p+1, \qquad \dim\mathcal F_x=r+1, \qquad x\in M \setminus (\Sigma \cup M_{\mathbf 0}),
\]
where $\Sigma$ is a $p$-dimensional invariant subvariety of $M\setminus M_{\mathbf 0}$, where $\mathcal F$ is not transversal to the contact distribution $\mathcal H$.
Furthermore, $\Sigma$ is the singular leaf of $\mathcal E$:
\begin{align*}
\dim\mathcal E_x = p, \qquad \mathcal F_x \subset \mathcal H_x, \qquad \dim\mathcal F_x=r+1, \qquad T_x\Sigma =\mathcal E_x, \qquad x\in\Sigma.
\end{align*}
Similarly, the zero locus $M_\mathbf 0$ can also be defined as a set of points where the foliation $\mathcal E$ is tangent to the contact distribution $\mathcal H$ and we have
\[
\dim\mathcal E_x=p, \qquad \dim\mathcal F_x=r, \qquad \mathcal F_x \subset \mathcal E_x\subset \mathcal H_x, \qquad \quad x\in M_{\mathbf 0}.
\]

This allows us to compare constructions of \cite{Jo} and \cite{Zung} within a unified geometric framework:
on $M$ we have integrability with respect to Theorem \ref{LIN},
on $M\setminus M_{\mathbf 0}$ we have toric contact integrability in the sense of \cite{Zung}, and on $M\setminus (\Sigma \cup M_{\mathbf 0})$ we have a complete pre-isotropic contact structure and noncommutative contact integrability in the sense of \cite{Jo}.

Note that in the cooriented case, when the line bundle is trivial, the Jacobi structure in the space of sections reduces to the standard Jacobi structure.
Therefore, we simultaneously treat contact systems on cooriented and non-cooriented contact manifolds.
This allows us to work with dissipative Hamiltonian systems where the Hamiltonian does not need to be preserved by the original Reeb vector field \cite{deLeon}.
For simplicity, we have considered the commutative case when $\Sigma=\emptyset$
(Theorem \ref{komutativna}, Section \ref{sec5}).
Recently, the statement analogous to item (ii) of  Theorem \ref{komutativna} was given in \cite{Col}, with the proof based on the symplectization of contact manifolds and the Liouville-Arnold theorem for exact symplectic manifolds for commuting integrals which are homogeneous of degree 1.

\section{Toric integrability and complete pre-isotropic contact structures}\label{sec2}

In this section,  following \cite{Jo, Zung}, we recall on the concepts of contact integrability on contact manifolds
and complete pre-isotropic contact structures.

\subsection{Toric integrability}\label{toric}

A general concept of the so-called non-Hamiltonain integrabity of the equation \eqref{eq2}, following \cite{Bo, Koz, Zu},
is given by the following conditions

\begin{itemize}

\item[(i)] The equations \eqref{eq2} have first integrals $f_1,\dots,f_p$, independent on an open dense
subset $M_{reg}\subset M$, which defines a proper submersion
\begin{equation*}\label{sub*}
\pi=(f_1,\dots,f_p)\colon M_{reg}\longrightarrow
W\subset \R^p.
\end{equation*}

\item [(ii)] There is an abelian Lie algebra
$\mathcal X=\Span\{X_0,X_1,\dots,X_{r}\}$ of symmetries,
\[
[X_i,X_j]=0, \qquad i,j=0,\dots,r,
\]
so that the fibers of $\pi$ are orbits of $\mathcal X$ ($r=2n-p$).

\end{itemize}

From (i) and (ii) we have:
$M_{reg}$ is foliated on $(r+1)$--dimensional tori with linear dynamics, which can be determined by quadratures.
Let us consider a connected component $M_{\mathbf c}^0$ of the invariant manifold $M_{\mathbf c}=\pi^{-1}({\mathbf c})\cap M_{reg}$.
There exists a $\pi$--invarinat toroidal neighborhood $U$ of $M_{\mathbf c}^0$ with Liouville coordinates
$(\varphi_0 (\text{mod }2\pi),\dots,\varphi_{r}(\text{mod }2\pi),z_1,\dots,z_p)$ such that invariant tori are the level sets of coordinates $z_1,\dots,z_p$
($z_i=z_i(f_1,\dots,f_p)$) and the vector fields $X_0,\dots,X_{r}$ are linear
\[
X_i=\sum_{j=0}^{r} \Omega_{ij}(z_1,\dots,z_p)\frac{\partial}{\partial\varphi_j}, \qquad i=0,\dots,r,
\]
where the matrix of frequencies $(\Omega_{ij})$ is invertible (see the construction of the angle coordinates in \cite{Ar}).
\footnote{Note: If we relax the condition that the algebra $\mathcal X$ is abelian by the condition that it is solvable, we still have the solvability of the equation by quadratures \cite{Koz, Koz2, CFG}.}

We have the natural induced action of $\mathbb T^{r+1}$ on $U$ by the translations in angle coordinates.
According to Theorem 2.2 \cite{Zung}
any $\mathcal X$--invariant tensor field $T$
is also invariant with respect to the torus action, motivating
the following natural definition of contact integrability (see Zung \cite{Zung}).

\begin{dfn} \label{defZung}
We say that the contact system \eqref{eq2} is
\emph{toric contact integrable} if it has a set of integrals and commuting contact symmetries satisfying conditions (i), (ii), and
\begin{itemize}

\item [(iii)] The symmetries $\mathcal X$ are contact: $X_k\in \mathfrak X(M,\mathcal H)$, $k=0,\dots,r$.

\end{itemize}
\end{dfn}

Then it follows that there exist a toroidal neighborhood $U$ of the invariant torus $M_{\mathbf c}^0$
with Liouville coordinates $(\varphi,z)$ and a local $\mathbb T^{r+1}$-invariant \emph{normal contact 1-form} $\alpha_U$ which defines $\mathcal H\vert_U$
\begin{equation}\label{NCF}
\alpha_U=\sum_{k=0}^r a_k(z)d\varphi_k +\sum_{i=1}^p b_i(z) dz_i,
\end{equation}
where $a_k(z),b_i(z)$ are smooth functions (see Proposition 4.2 \cite{Zung}).

From the conditions (i) and (ii) we do not have restriction on the dimension of invariant tori.
However, with the addition of the assumption (iii), from \eqref{NCF} we get that
the two-form $d\alpha_U$ vanishes on invariant tori implying
the inequality  $r=2n-p\le n$  ($\alpha_U\wedge(d\alpha_U)^n\ne0$).

Note that under the additional general assumption that $\mathcal X$ is transversal to $\mathcal H$, the normal contact form \eqref{NCF} can be further simplified (see equation \eqref{canonical**} below).

\subsection{Contact Hamiltonian vector fields and Jacobi brackets}

In what follows we mainly follow the notation of Libermann and Marle \cite{LM}.
On a cooriented contact manifold $(M,\alpha)$ we have
a $\mathbb R$-linear isomorphism $\Phi_\alpha$ between contact vector fields and the space of smooth $\mathbb R$-valued functions $\mathcal A(M)$:
\begin{equation}\label{isoA}
\Phi_\alpha\colon \mathfrak X(M,\mathcal H)\to \mathcal A(M), \quad \Phi_\alpha(X)=\alpha(X).
\end{equation}

The inverse image of $f\in\mathcal A(M)$ is called the \emph{contact Hamiltonian vector field} $X_f$ with the Hamiltonian function $f$.
The mapping \eqref{isoA} is a Lie algebra isomorphism, where
$\mathfrak X(M,\mathcal H)$ is endowed with the usual bracket of vector fields and $\mathcal A(M)$ with the \emph{Jacobi bracket} defined by
\begin{equation*}
[f_1,f_2]_\alpha=\Phi_\alpha([X_{f_1},X_{f_2}]) \quad \Longleftrightarrow \quad X_{[f_1,f_2]_\alpha}=[X_{f_1},X_{f_2}].
\end{equation*}
%(we follow the sign convection given in \cite{LM}).

The contact Hamiltonian vector fields can be described as follows.
The \emph{Reeb vector field} $Z_\alpha$ on $(M,\alpha)$ is the vector field transverse to $\mathcal H$, defined by
$$
\alpha(Z_\alpha)=1, \qquad i_{Z_\alpha} d\alpha=0.
$$
The tangent bundle $TM$ and the cotangent bundle $T^*M$ are decomposed into
\begin{equation}
TM=\mathcal Z \oplus \mathcal H, \qquad T^*M=\mathcal Z^0\oplus
\mathcal H^0, \label{decomposition}
\end{equation}
where $\mathcal Z=\R Z_\alpha$ is the kernel of
$d\alpha$, $\mathcal Z^0$ and $\mathcal H^0=\R\alpha$ are the annihilators of $\mathcal Z$ and $\mathcal H$ respectively.
According to \eqref{decomposition} we have decompositions of vector fields and 1-forms
\begin{equation}
X=\alpha(X)Z_\alpha+\hat X, \qquad \eta=\eta(Z_\alpha)\alpha+\hat \eta,
\label{decomp*}
\end{equation}
where $\hat X\in\Gamma(\mathcal H)$ and $\hat\eta\in\Gamma(\mathcal Z^0)$.
The mapping $\alpha^\flat: X \mapsto -i_X d\alpha$ is an isomorphism between $\Gamma(\mathcal H)$ and $\Gamma(\mathcal Z^0)$,
whose inverse is denoted by $\alpha^\sharp$.

Now the Hamiltonian vector field $X_f$ takes the form
$$
X_f=\Phi_\alpha^{-1}(f) = f Z_\alpha+\alpha^\sharp(\hat{df})
=f Z_\alpha+\alpha^\sharp({df}-df(Z_\alpha)\alpha)
$$
and
$
\mathcal L_{X_f}\alpha=Z_\alpha(f)\alpha.
$
Note that $Z_\alpha=X_1$. Furthermore, the Jacobi bracket reads:
\begin{equation}
\label{JBN}
[f_{1},f_{2}]_\alpha = d\alpha(X_{f_1},X_{f_{2}})+f_{1}{Z_\alpha}(f_{2})-f_{2}Z_\alpha(f_{1})= X_{f_{1}}(f_{2})-f_{2}Z_\alpha (f_{1}).
\end{equation}

\begin{exm}\label{canonical} The basic example of a  co-oriented contact manifold is a linear space
$\R^{2n+1}(q_0,q_1,\dots,q_n,p_1,\dots,p_n)$ with a canonical contact form
$\alpha=dq_0+\sum_{i=1}^n p_i dq_i$. Then the Reeb vector field $Z_\alpha=\frac{\partial}{\partial q_0}$
and the contact Hamiltonian equations for a Hamiltonian $f$ are
\begin{align}\label{dissipative}
\dot q_0=f-\sum_{i=1}^n p_i\frac{\partial f}{\partial p_i},\quad
\dot q_i=\frac{\partial f}{\partial p_i},\quad \dot p_i=-\frac{\partial f}{\partial q_i}+p_i\frac{\partial f}{\partial q_0}, \quad i=1,\dots,n.
\end{align}

The Jacobi bracket reads:
\begin{equation*}
[f,g]= \sum_{i=1}^n\big(\frac{\partial f}{\partial p_i}\frac{\partial g}{\partial q_i}
-\frac{\partial g}{\partial p_i}\frac{\partial f}{\partial q_i}\big)+\frac{\partial g}{\partial q_0}\big(f-\sum_{i=1}^n p_i\frac{\partial f}{\partial p_i}\big)-\frac{\partial f}{\partial q_0}\big(g-\sum_{i=1}^n p_i\frac{\partial g}{\partial p_i}\big).
\end{equation*}

According to Darboux's theorem for contact manifolds, every contact manifold $(M,\mathcal H)$ has local canonical coordinates $(q_0,q_1,\dots,q_n,p_1,\dots,p_n)$ in which the local contact form $\alpha$ has the canonical form (see \cite{LM}).
Note that from the point of view of Hamiltonian dynamics, the equations \eqref{dissipative} describe a class of dissipative Hamiltonian systems, see e.g. \cite{deLeon}.
\end{exm}

\subsection{Pre-isotropic foliations and integrability}\label{PIF}

Let $\mathcal F$ be a foliation on a cooriented contact manifold
$(M^{2n+1},\alpha)$. The {\it pseudo-orthogonal distribution}
$\mathcal F^\perp$ locally generated by the Reeb vector field $Z_\alpha$ and the contact
Hamiltonian vector fields, which correspond to the first integrals of $\mathcal F$. A foliation $\mathcal F$ is called \emph{$\alpha$-complete} if for each pair $f_1,f_2$ of first integrals of $\mathcal F$ (where
$f_i$ can be a constant), the bracket $[f_1,f_2]_\alpha$ is also a first integral of $\mathcal F$ (possibly a constant). A foliation $\mathcal F$ that contains the
Reeb vector field $Z_\alpha$ is $\alpha$-complete if the pseudo-orthogonal subbundle $\mathcal F^\perp$ is integrable and defines a foliation which is also $\alpha$-complete and
$(\mathcal F^\perp)^\perp=\mathcal F$ \cite{L}. Then for each pair of integrals $f,g$ of $\mathcal F$ and $\mathcal F^\perp$,
respectively $[f,g]_\alpha=0$ \cite{L}. Furthermore, we say that a foliation $\mathcal F$ is {\it pre-isotropic} if

\begin{itemize}

\item[(i)] $\mathcal F$ is transversal to $\mathcal H$

\item[(ii)] $\mathcal G=\mathcal F \cap\mathcal H$ is an isotropic subbundle of the symplectic vector bundle $(\mathcal H,d\alpha)$:
\[
\mathcal G_x\subset \mathcal G_x^{d\alpha\vert_\mathcal H}=\{\xi\in\mathcal H_x\,\vert\, d\alpha(\xi,\mathcal G_x)=0\}.
\]
\end{itemize}

Note that condition (ii) is equivalent to the condition that $\mathcal
G=\mathcal F\cap \mathcal H$ is a foliation (Lemma 3.1 \cite{Jo}).
The geometric framework for non-commutative integrability on cooriented contact manifolds is motivated by the following statement (Theorem 3.2 \cite{Jo}).

\begin{prop}\label{STAV} Let $\mathcal F$ be a pre-isotropic foliation containing the Reeb vector field $Z_\alpha$. Then:

(i) We have a flag of the distribution:
\begin{equation}\label{flag}
\mathcal G=\mathcal F \cap \mathcal H \,\,\subset\,\, \mathcal F
\,\,\subset \,\,\mathcal E=\mathcal G^\perp=\mathcal F^\perp.
\end{equation}

%On the other hand, if $\mathcal F$ is a foliation containing the Reeb vector field $Z_\alpha$ and \eqref{flag} holds, then $\mathcal F$ is a
%pre-isotropic foliation.

(ii) The foliation $\mathcal F$ is
$\alpha$-complete if and only if $\mathcal E$ is completely integrable. Let us assume that $\mathcal E$ is integrable and that
$f_1,\dots,f_p$ and $y_1,\dots,y_r$, $2n-p=r$ are arbitrary sets of local integrals of $\mathcal F$ or $\mathcal E$. Then:
$[f_i,y_j]_\alpha=0$, $[y_j,y_k]_\alpha=0$, $[f_i,1]_\alpha=0$, $[y_i,1]_\alpha=0$.

(iii) Each leaf of a $\alpha$-complete pre-isotropic foliation
$\mathcal F$ and each leaf of the corresponding isotropic foliation $\mathcal G$ has an affine structure.
\end{prop}

Since $\mathcal F\subset \mathcal E$, we can take local integrals $f_1,\dots,f_p$ of $\mathcal F$ such that
$f_1,\dots,f_r$ are integrals of $\mathcal E$. Then
\begin{equation}\label{involucija}
[1,f_i]_\alpha=0, \quad [f_i,f_j]_\alpha=0, \quad i=1,\dots,p, \quad j=1,\dots, r, \quad r+p=2n.
\end{equation}

Consider the contact Hamiltonian equation
\begin{equation}\label{eq_f}
\dot x=X_h
\end{equation}
with the Hamiltonian $h\in\mathcal A(M)$. For the contact Hamiltonain vector fields that commute with the Reeb vector field $[X_h,Z_\alpha]_\alpha=0$
(or equivalently for the Hamiltonians $H$ which are integrals of the Reeb flow $Z_\alpha(h)=0$), we have the following theorem on non-commutative integrability (see Theorem 5.2 and Remark 5.3 in \cite{Jo}).

\begin{thm}\label{teo1}
Suppose we have a collection of integrals
\[
\pi=(f_1,f_2,\dots,f_{p})\colon M \longrightarrow \R^p
\]
of the equation \eqref{eq_f} with the contact Hamiltonian either
$h=f_1$ or $h=1$, so that \eqref{involucija} holds.
Let $M_{reg}$ be the open subset of $M$ such that $df_1 \wedge \dots \wedge df_{p}\ne 0$.
\footnote{Note that, since $df_i(Z_{\alpha})=0$, the condition $df_1 \wedge \dots \wedge df_{p}\ne 0$
is equivalent to $df_1 \wedge \dots \wedge df_{p}\wedge \alpha\ne 0$ (see Lemma \ref{LEMA}).}
Then the invariant manifolds of the mapping $\pi$
 define $\alpha$--complete pre-isotropic foliations $\mathcal F$ on $M_{reg}$ and
\[
\mathcal X=\Span\{Z_{\alpha}, X_{f_1},\dots,X_{f_r}\}
\]
is an abelian Lie algebra of contact symmetries which has the fibers of $\pi$ as orbits.
Let $M_{\mathbf c}^0$ be a compact connected component of $\pi^{-1}(\mathbf c)\cap M_{reg}$.
Then $M_{\mathbf c}^0$ is diffeomorphic to a $(r+1)$--dimensional torus
$\mathbb T^{r+1}$. There exists an open $\pi$-invariant neighborhood $U$ of $M_{\mathbf c}^0$ with Liouville coordinates
\begin{equation}\label{action-angle}
(\varphi,I,q,p)=(\varphi_0,\varphi_1,\dots,\varphi_r,I_1,\dots,I_r,q_1,\dots,q_{s},p_1,\dots,p_s),
\quad 2s=p-r,
\end{equation}
so that the contact form $\alpha$ has the following canonical form
\begin{equation}\label{canonical*}
\alpha=I_0d\varphi_0+I_1d\varphi_1+\dots+I_rd\varphi_r+p_1dq_1+\dots+p_{s}dq_{s},
\end{equation}
with $I_k=I_k(f_1,\dots,f_r)$, $k=0,\dots,r$, $q_i=q_i(f_1,\dots,f_p)$, $p_i=p_i(f_1,\dots,f_p)$, $i=1,\dots,s$.
The flow of $X_h$ on invariant tori is quasi-periodic
\begin{equation}\label{namotavanje}
(\varphi_0,\varphi_1,\dots,\varphi_r) \longmapsto
(\varphi_0+t\omega_0,\varphi_1+t\omega_1,\dots,\varphi_r+t\omega_r).
\quad t\in\R,
\end{equation}
The Hamiltonian $h$ and the frequencies $\omega_0,\dots,\omega_r$ depend only on $I_1,\dots,I_r$.
\end{thm}

In analogy to the non-commutative integrability in Hamiltonian dynamics (see \cite{N, MF}), we denote the coordinates $(\varphi,I,q,p)$ given in Theorem \ref{teo1}
 as (generalised) \emph{contact action-angle} coordinates of the contact form $\alpha$.
$\mathcal F\vert_U$ is locally defined by the integrals $I_1,\dots,I_r,q_1,\dots,q_{s},p_1,\dots,p_s$, while the integrals of the pseudo-orthogonal foliation $\mathcal E\vert_U$ are $I_1,\dots,I_r$.
As in the Hamiltonian case, let $\gamma_k(T)$ be a cycle that is homologous to the trajectories of the field
${\partial}/\partial\varphi_{k}$ restricted to an arbitrary invariant torus $T$ within $U$. Then
\begin{equation}\label{actions}
I_k\vert_T=\frac{1}{2\pi}\int_{\gamma_k(T)} \alpha, \qquad k=0,\dots,r.
\end{equation}

If $r=n$ ($s=0$) we have commutative integrability \cite{BM}.

\subsection{Complete pre-isotropic contact structures}

Let us now consider a general case in which $\mathcal H$ need not be cooriented.
Although we do not have a Reeb vector field, the concept of pre-isotropic and isotropic foliations $\mathcal F\subset TM$, $\mathcal G\subset \mathcal H$
are well-defined, since the vector subspace $\mathcal V_x\subset \mathcal H_x$ at $x\in U\cap V$ is isotropic (Lagrangian, coisotropic) with respect to
$d\alpha_U\vert_{\mathcal H}$,
\[
\mathcal V_x\subset \mathcal (\mathcal V_x)^{d\alpha_U\vert_{\mathcal H}}\subset \mathcal H_x \quad (\mathcal V_x = \mathcal (\mathcal V_x)^{d\alpha_U\vert_{\mathcal H}}, \quad
\mathcal H_x \supset \mathcal V_x \supset \mathcal (\mathcal V_x)^{d\alpha_U\vert_{\mathcal H}}),
\]
then and only then it is isotropic (Lagrangian, coisotropic) with respect to $d\alpha_V\vert_{\mathcal H}$, where $\alpha_U$ and $\alpha_V$ are contact forms that determine $\mathcal H$ on $U$ and $V$ respectively.

Let us consider the {contact equation} \eqref{eq2}
on a $(2n+1)$-dimensional contact manifold $(M,\mathcal H)$.
In \cite{Jo} we have proposed the following definition

\begin{dfn}\label{defJov}
We say that the contact system \eqref{eq2} is
\emph{contact non-commutative integrable} if there exists an open dense set $M_{reg}$, an Abelian Lie algebra of contact symmetries 
$\mathcal X=\Span\{X_0,\dots,X_r\}$ containing $X$, and a proper submersion
\begin{equation}\label{sub*}
\pi: M_{reg}\to W
\end{equation}
to a $p$-dimensional manifold $W$ such that
$(M_{reg},\mathcal H,\mathcal X)$ is a \emph{complete pre-isotropic contact structure}. This means that
\begin{itemize}
\item[(a)] The foliation $\mathcal F$, defined by the fibers of $\pi$, $\mathcal F_x=\ker d\pi\vert_x$
is pre-isotropic, i.e. it is transversal to $\mathcal H$ and $\mathcal G=\mathcal F\cap \mathcal H$ is an isotropic subbundle of $\mathcal H$;
\item[(b)] $\mathcal X$ has the fibers of $\pi$ as orbits.
\end{itemize}
\end{dfn}

Let us assume that the contact system \eqref{eq2} is contact-integrable in a non-commutative sense.
Then the connected components $M_{\mathbf c}^0$ of the invariant manifolds $M_{\mathbf c}=\pi^{-1}({\mathbf c})\cap M_{reg}$ have an open, $\pi$-invariant neighborhood $U$ with a $\mathcal X$--invariant local contact form $\alpha_U$, so that the pre-isotropic foliation $\mathcal F\vert_U$ is
$\alpha_U$-complete (see Theorems 4.1 and 5.1 in \cite{Jo}). Therefore, in $U$ the fibration $\pi$ is defined by integrals
\begin{equation}\label{integrali}
\pi=(f_1,f_2,\dots,f_{p})\colon U \longrightarrow \R^p
\end{equation}
satisfying the conditions of  Theorem \ref{teo1}, where either $X$ is the Reeb vector field on $(U,\alpha_U)$, or $X=X_{f_1}$. We have contact action angle coordinates $(\varphi,I,q,p)$ in which $\alpha_U$ takes the normal form
\begin{equation}\label{canonical**}
\alpha_U=I_0d\varphi_0+I_1d\varphi_1+\dots+I_rd\varphi_r+p_1dq_1+\dots+p_{s}dq_{s}.
\end{equation}

Note that Zung in Theorem 4.3 \cite{Zung} proves the existence of a contact normal form analogous to \eqref{canonical**}
assuming that the contact symmetries $\mathcal X$ are transversal to the contact distribution $\mathcal H$, instead of assuming that the invariant foliation $\mathcal F$ is pre-isotropic.

The two approaches are equivalent. In fact, we have the following lemma

\begin{lem}\label{lemT}
The conditions (a) and (b) are equivalent to the conditions
\begin{itemize}
\item[(a)'] The contact symmetries $\mathcal X$ are transversal to $\mathcal H$;

\item[(b)\,] $\mathcal X$ has the fibers of $\pi$ as orbits.
\end{itemize}
\end{lem}

\begin{proof} It suffices to prove that the conditions (a)', (b) imply (a), (b).

Let $x_0\in M_{reg}$ and let $\tilde\alpha_V$ be a local contact form in a neighborhood $V\subset M_{reg}$ of $x_0$ that defines $\mathcal H\vert_V$.
Since $\mathcal X$ is transversal to $\mathcal H$, we can possibly take a smaller neighborhood $V$ without loss of generality,
we assume that $\tilde\alpha_V(X_{0})\vert_V\ne 0$. Then $X_{0}$ is the Reeb vector field $Z_{\alpha_V}$ of the contact form
$\alpha_V=1/\tilde\alpha_V(X_{0})\tilde\alpha_V$.
Define functions $g_k=\alpha_V(X_k)$, $k=1,\dots,r$. From the commutativity of the vector fields it follows that
\[
[g_k,g_l]_{\alpha_V}=0, \qquad [g_k,1]_{\alpha_V}=0, \qquad k,l=1,\dots,r.
\]

Thus $Z_{\alpha_U}(g_k)=0$ and from \eqref{JBN} we get
\[
0=d\alpha_V(X_{g_k},X_{g_l})=d\alpha_V(\alpha_V^\sharp({dg_k}),\alpha_V^\sharp({dg_l})), \qquad k,l=1,\dots,r.
\]

Therefore, $\mathcal G\vert_V=\mathcal F\vert_V\cap\mathcal H\vert_V=\Span\{\alpha_V^\sharp({dg_k})\vert k=1,\dots,r\}$
is isotropic, i.e. $\mathcal F\vert_V$ is a pre-isotropic foliation on $(V,\alpha_V)$.
\end{proof}

Recall that the condition that $\mathcal G=\mathcal F\cap \mathcal H$ is an isotropic subbundle of $\mathcal H$ is equivalent to the condition that $\mathcal G$ is a foliation (see Subsection \ref{PIF}).

In the case of maximal dimension of the invariant tori ($r=n$), i.e. when
$\mathcal G=\mathcal F\cap \mathcal H$ is a
Lagrangian subbundle of $\mathcal H$, the definition \ref{defJov} coincides with the definition of contact integrability given by
Banyaga and Molino \cite{BM}. Note that in contact geometry, Lagrangian foliations are usually refer as
\emph{Legendre} foliations. In this case we say that the foliation $\mathcal F$ is \emph{pre-Legendrian}.

Examples of contact-integrable systems
with both pre-Legendrian and pre-isotropic invariant manifolds can be found in \cite{KT, Boyer, JJ1, JJ2, JJ3, SS}.

The obstructions to the existence of global contact action angle coordinates are studied in \cite{BM} for the case of commutative integrability and in \cite{Jo, SS} for the case of non-commutative integrability.

In the next two sections, we will develop a geometric framework that deals with the symmetries which need not be transverse to the contact distribution.
Simultaneously, in the cooriented case, that will include the case when the Hamiltonian is not the first integral.

\section{The contact line bundles and integrability}\label{sec3}

\subsection{Contact structure and the associated line bundle}

Following Gray \cite{Gr}, we will use a basic sheaf theory to describe the contact structure.
Let $\mathcal A$ be the sheaf of smooth real functions, and $\mathcal A^*$ the sheaf of invertible real functions on $M$.
The line bundle $L$ can be associated to a contact manifold $(M,\mathcal H)$ as follows.
Let $\{(U,\alpha_U)\}$ be a covering of $M$ with local contact forms defining $\mathcal H$.
Then for each non-empty intersection $U\cap V$ holds
\begin{equation}\label{localContact}
\alpha_U=g_{UV}\alpha_V
\end{equation}
for a smooth function $g_{UV}$, where $g_{UV}(x)\ne 0, x\in U\cap V$, i.e. $g_{UV}\in \mathcal A^*(U\cap V)$.
Then $\{g_{UV}\}$ are transition functions of the line bundle $L\in H^1(M,\mathcal A^*)$. Indeed,
It follows from the definition that whenever $U\cap V\cap W \ne\emptyset$, we have
\[
g_{UV}(x)g_{VW}(x)g_{WU}(x)=1, \quad x\in U\cap V\cap W, \quad \text{i.e,} \quad \delta\{g_{UV}\}=0.
\]

If we take other local contact forms $\alpha'_U=f_U\alpha_U$, $f\in\mathcal A^*(U)$ for a given covering $\{U\}$,
then we have on $U\cap V$
\[
\alpha'_U=g'_{UV} \alpha'_V, \quad g'_{UV}=g_{UV}\frac{f_V}{f_U}, \quad \text{i.e,} \quad \{g'_{UV}\}-\{g_{UV}\}=\delta\{f_U\}.
\]
Thus $\{g'_{UV}\}$ and $\{g_{UV}\}$ define the same line bundle over $M$ (see \cite{GH}).

A smooth section $s\in\Gamma(L)$ is a collection of smooth functions $\{s_U\}$ on the cover $\{U\}$ such that
\[
s_U=g_{UV} s_V.
\]
and the above-mentioned collection of local contact forms $\{\alpha_U\}$ can be regarded as an element $\alpha$ of $\Omega^1(L)$ (the space of line bundle-valued 1-forms \cite{GH}). We can always consider the case with local coordinates of the covering $\{(U,\alpha_U)\}$ in which the representatives $\alpha_U$ of the contact structure $\alpha\in \Omega^1(L)$ have the canonical form described in Example \ref{canonical} (see \cite{Gr}).

Now there is a $\mathbb R$-linear isomorphism $\Phi$ between contact vector fields and the space of smooth sections
 of $L$:
\begin{equation}\label{iso2}
\Phi\colon \mathfrak X(M,\mathcal H)\to \Gamma(L), \qquad \Phi(X)_U=\Phi_{\alpha_U}(X\vert_U)=\alpha_U(X\vert_U).
\end{equation}

The inverse image of a section $s\in\Gamma(L)$ is called the \emph{contact Hamiltonian vector field} $X_s$ with the Hamiltonian $s$.
As above, the mapping \eqref{iso2} is a Lie algebra isomorphism, where on
$\mathfrak X(M,\mathcal H)$ is the usual bracket and the {\it Jacobi bracket} on $\Gamma(L)$ is defined by
\begin{equation*}
[s_1,s_2]=\Phi([X_{s_1},X_{s_2}]) \quad \Longleftrightarrow \quad X_{[s_1,s_2]}=[X_{s_1},X_{s_2}].
\end{equation*}

Now the Hamiltonian vector field and the Jacobi bracket on $\Gamma(L)$ are given locally by:
\begin{align*}
& X_s\vert_U =\Phi^{-1}_{\alpha_U}(s_U)=s_U Z_{\alpha_U}+\alpha_U^\sharp({ds_U}-ds_U(Z_{\alpha_U})\alpha_U) \quad (\mathcal L_{X_s\vert_U}\alpha_U=Z_{\alpha_U}(s_U)\alpha_U), \\
& [s_1,s_2]_U = [s_{1U},s_{2U}]_{\alpha_U}= X_{s_{1U}}(s_{2U})-s_{2U}Z_{\alpha_U} (s_{1U}),
\end{align*}
while for the non-empty intersection points $U\cap V$ applies
\begin{align*}
& X_s\vert_U=\Phi^{-1}_{\alpha_U}(s_U)=\Phi^{-1}_{g_{UV}\alpha_V}(g_{UV}s_V)=\Phi^{-1}_{\alpha_V}(s_V)=X_s\vert_V,\\
& [s_{1U},s_{2U}]_{\alpha_U}=[g_{UV}s_{1V},g_{UV}s_{2v}]_{\alpha_U}=g_{UV}[s_{1V},s_{2V}]_{\alpha_V}.
\end{align*}

\subsection{From contact symmetries to integrals}\label{CS->INT}

Consider the contact Hamiltonain equation
\begin{equation}
\dot x=X_h, \qquad h\in\Gamma(L).
\label{eq2h}
\end{equation}
A section $s\in\Gamma(L)$ is called a \emph{symmetry} of the Hamiltonian $h$ if $[h,s]=0$.

If $X_1$ and $X_2$ are contact symmetries, then it is $\lambda_1 X_1+\lambda_2 X_2$, $\lambda_1,\lambda_2\in\R$. Let us consider the vector space of contact symmetries
$\mathcal Y=\Span \{X_0,\dots,X_p\}\subset \mathfrak X(M,\mathcal H)$,
and the associated vector space
\begin{equation}\label{simetrije}
\mathbf Y=\Span\{s_0=\Phi(X_0),\dots,s_p=\Phi(X_p)\}=\Phi(\mathcal Y)\subset\Gamma(L)
\end{equation}
of the symmetries of the Hamiltonian $h$.
In the following, we will describe a natural construction of the corresponding first integrals.

Let $s\in\Gamma(L)$. Consider the open submanifold
\[
M_s=\{x\in M\, \vert\, s(x)\ne 0\}.
\]
On $M_s$ we have globally well-defined contact 1-form $\alpha_s=\alpha/s$:
\[
\alpha_{sU}(x)=\frac{\alpha_U(x)}{s_U(x)}=\frac{g_{UV}(x)\alpha_V(x)}{g_{UV}(x)s_V(x)}=\alpha_{sV}(x), \qquad x\in U\cap V\cap M_s.
\]

Let us also consider the mapping
\[
\phi_s\colon \Gamma(L)\to \mathcal A(M_s), \qquad \phi_s(l)=l/s.
\]

It is easy to prove that this is a homomorphism of the Lie algebras $(\Gamma(L),[\cdot,\cdot])$ and $(\mathcal A(M_s),[\cdot,\cdot]_{\alpha_s})$, which maps the contact Hamiltonian vector field
$X_s$ to the Reeb vector field of $(M_s,\alpha_s)$:
\begin{align*}
& \phi_s([l_1,l_2])=[\phi_s(l_1),\phi_s(l_2)]_{\alpha_s},\\
& X_s\vert_{M_s}=\Phi^{-1}(s)\vert_{M_s}=Z_{\alpha_s}=\Phi^{-1}_{\alpha_s}(1),\\
& X_l\vert_{M_s}=\Phi^{-1}(l)\vert_{M_s}=\Phi^{-1}_{\alpha_s}(\phi_s(l)).
\end{align*}

\begin{lem}\label{pomocnaLema}
Let us assume that $s\in\Gamma(L)$ is a symmetry of the Hamiltonian $h$: $[s,h]=0$. Then $M_s$ and its complement, the zero divisor
\[
(s)=M\setminus M_s=\{x\in M\, \vert\, s(x)=0\},
\]
are invariant with respect to the contact flow \eqref{eq2h}.
\end{lem}

\begin{proof}
Consider a local chart $(U,\alpha_U)$, with $s_U$ and $h_U$ being the local representatives of $s$ and $h$.
For $x\in U\cap (s)$ we then have $s_U(x)=0$ and
\[
0=[h_{U},s_{U}]_{\alpha_U}(x) = X_{h}(s_{U})\vert_x-s_{U}Z_{\alpha_U}(h_U)\vert_x=X_{h}(s_U)\vert_x.
\]
Therefore, (s) is invariant under the flow of $X_{h}$.
\end{proof}

\begin{thm}\label{pomocna}
Consider symmetries \eqref{simetrije} of the contact Hamiltonian equations \eqref{eq2h}.

(i) The zero locus of the symmetries
\[
M_{\mathbf 0}=\{x\in M\,\vert\, s(x)=0, \, s\in \mathbf Y\}
\]
is an invariant variety of the system \eqref{eq2h}.

(ii) Let $s\in \mathbf Y$. The functions
\[
\mathbf Y_s=\{\phi_s(l)=l/s\,\vert\, l\in\mathbf Y\}\subset \mathcal A(M_s)
\]
are the first integrals  restricted to $M_s$.

(iii) The symmetries $\mathbf Y$ determine well-defined mappings
\begin{equation}\label{PI}
\pi\colon M\setminus M_{\mathbf 0} \longrightarrow \mathbb{RP}^p, \qquad \pi(x)\vert_U=[s_{0U}(x),\dots,s_{pU}(x)],
\end{equation}
where $[\cdot,\cdots,\cdot]$ are homogeneous coordinates on $\mathbb{RP}^p$.
Let $M_{reg}$ be an open set of $M\setminus M_\mathbf 0$, where the rank of the differential $d\pi$ is $p$.
The sets $M_{reg}$,
\[
M_{reg,i}=M_{reg}\cap M_{s_i}, \qquad i=0,\dots,p,
\]
and the $(2n+1-p)$-dimensional fibers $\pi^{-1}(c)\cap M_{reg}$ ($c\in \pi(M_{reg})$) are invariant under the contact flow \eqref{eq2h}.
\end{thm}

Let $\mathcal E\subset TM$ be a distribution generated by contact Hamiltonian vector fields of symmetries $\mathbf Y$:
\[
\mathcal E_x=\Span\{X_{s_0}(x),\dots,X_{s_p}(x)\}.
\]
Note that the zero set $M_\mathbf 0$ can also be defined as a set of points for which the distribution $\mathcal E$ is a subset of the contact distribution:
\begin{equation}\label{zeroL}
M_\mathbf 0=\{x\in M\,\vert\, \mathcal E_x\subset \mathcal H_x\},
\end{equation}
i.e. at the points outside $M_\mathbf 0$, the space of contact symmetries $\mathcal Y$ is transverse to the contact distribution $\mathcal H$.

The construction of the mapping \eqref{PI} by using the sections is standard in the theory of line bundles (see e.g. \cite{GH}).
We say that $\pi$ is the \emph{momentum mapping} associated with the symmetries $\mathbf Y$.

\begin{proof}[Proof of  Theorem \ref{pomocna}]
Item (i) follows from Lemma \ref{pomocnaLema}.

\medskip

(ii) If $s=h$ then $Z_{\alpha_h}=X_h\vert_{M_h}$, and $[h,l]=0$ implies $[1,\phi_h(l)]_{\alpha_h}=0$, i.e
\[
Z_{\alpha_h}(\phi_h(l))=0.
\]
Otherwise, for $s \ne h$, the involution $[h,s]=0$ implies $[\phi_s(h),1]_{\alpha_s}=0$ and $Z_{\alpha_s}(\phi_s(h))=0$, i.e,
\[
0=[\phi_s(h),\phi_s(l)]_s=X_{h}(\phi_s(l))-\phi_s(l)Z_s (\phi_s(h))=X_{h}(\phi_s(l)).
\]

\medskip

(iii) Let $y=(y_0,\dots,y_p)$ be the coordinates on $\R^{p+1}$ and $[y]=[y_0,\dots,y_p]$ the corresponding homogeneous coordinates on $\mathbb{RP}^p=(\R^{p+1}\setminus\{0\})/\R$.
 Obviously, the mapping $\pi$ is well-defined:
\[
[s_{0U}(x),\dots,s_{pU}(x)]=[g_{UV}(x)s_{0V}(x),\dots,g_{UV}(x)s_{pV}(x)]=[s_{0V}(x),\dots,s_{pV}(x)],
\]
$x\in U\cap V\setminus M_\mathbf 0$. For $x\in M_{s_i}$ we can consider a local coordinate chart $\vartheta_i$ of $[y]=\pi(x)\in\mathbb{RP}^p$,
\[
\vartheta_i\colon U_i\longrightarrow \mathbb R^p, \quad \vartheta_i([y_0,\dots,y_p])=\big(\frac{y_0}{y_i},\dots,\frac{y_{i-1}}{y_i},\frac{y_{i+1}}{y_i}, \dots, \frac{y_p}{y_i}\big),
\]
where $\{U_i\}_{i=0}^p$ is the standard covering of the projective space:
\[
U_i=\mathbb{RP}^p \setminus \{y_i=0\}, \qquad i=0,\dots,p.
\]

Then the rank of $d\pi(x)$, $[y]=\pi(x)\in U_i$ is equal to the rank $d\pi_i(x)$, where
\begin{equation}\label{PIi}
\pi_i=\vartheta_i\circ \pi=(\phi_{s_i}(s_0),\dots,\phi_{s_i}(s_{i-1}),\phi_{s_i}(s_{i+1}),\dots,\phi_{s_i}(s_p))\colon M_{s_i}\longrightarrow \R^p.
\end{equation}

The functions $\phi_{s_i}(s_k)=s_k/s_i$ are integrals of the flow (see item (ii)). It follows that
\[
\mathcal L_{X_h}(d(s_k/s_i))=i_{X_h} d^2(s_k/s_i)+d(X_h(s_k/s_i))=0,
\]
and the $p$-form
\begin{equation}\label{omegaI}
\omega_{s_i}=d(s_0/s_i)\wedge\dots\wedge d(s_{i-1}/s_{i})\wedge d(s_{i+1}/s_{i})\wedge\dots\wedge d(s_{p}/s_{i})
\end{equation}
is also invariant along the flow:
\begin{align*}
\mathcal L_{X_h}(\omega_{s_i})=& \big(\mathcal L_{X_h}d(s_0/s_i)\big)\wedge\dots\wedge d(s_{i-1}/s_{i})\wedge d(s_{i+1}/s_{i})\wedge\dots\wedge d(s_{p}/s_{i})+\dots+\\
& + d(s_0/s_i)\wedge\dots\wedge d(s_{i-1}/s_{i})\wedge d(s_{i+1}/s_{i})\wedge\dots\wedge \big(\mathcal L_{X_h}d(s_{p}/s_{i})\big)=0.
\end{align*}

Since $M_{reg,i}=M_{reg}\cap M_{s_i}$ can be defined as the subset of $M_{s_i}$ in which $\omega_{s_i}\ne 0$,
it is invariant along the flow.
In particular, we have obtained that the $(2n+1-p)$-dimensional submanifolds $\pi^{-1}(c)\cap M_{reg}$ ($c\in \pi(M_{reg})$) and $M_{reg}$ are also invariant.
\end{proof}

\subsection{The dual of the tautological line bundle}

Note that we can use $(M_{s_i},\alpha_{s_i})$ as a covering of $M\setminus M_\mathbf 0$ with local contact forms defining the contact distribution $\mathcal H$.
Since on $M_{s_i}\cap M_{s_j}$ then applies
\[
\alpha_{s_i}=\frac{1}{s_i}\alpha=\frac{s_j}{s_i}\alpha_{s_j},
\]
the transition functions of the line bundle $L$ are given by:
\[
g_{M_{s_i},M_{s_j}}=\frac{s_j}{s_i}.
\]

Therefore, the line bundle $L$ restricted to $M\setminus M_\mathbf 0$ is equal to $\pi^*\mathcal O(1)$ - the pull-back of the dual of the tautological line bundle $\mathcal O(1)$ over $\mathbb{RP}^p$. Recall that
$\mathcal O(1)$ is defined by the transition functions
\[
g_{U_i,U_j}([y]):=\frac{y_j}{y_i}, \qquad [y]\in U_i\cap U_j,
\]
with respect to the standard covering $\{U_i\}_{i=0}^p$ of $\mathbb{RP}^p$.
A section $l\in\Gamma(\mathcal O(1))$ is locally given by smooth functions $l_{U_i}\in\mathcal A(U_i)$ that satisfy
\[
l_{U_i}(\vartheta_i[y])= g_{U_i,U_j}([y])\, l_{U_j}(\vartheta_j[y]),
\]
i.e,
\[
y_i \, l_{U_i}\big(\frac{y_0}{y_i},\dots,\frac{y_{i-1}}{y_i},\frac{y_{i+1}}{y_i}, \dots, \frac{y_p}{y_i}\big)=y_j\, l_{U_j}\big(\frac{y_0}{y_j},\dots,\frac{y_{j-1}}{y_j},\frac{y_{j+1}}{y_j}, \dots, \frac{y_p}{y_j}\big).
\]
Thus $l$ can be considered as a smooth homogeneous degree 1 function in the variables $y=(y_0,\dots,y_p)$, which is defined for $y\ne 0$.

\subsection{Contact symmetries and non-commutative integrability}

We use the notation of  Theorem \ref{pomocna}.

\begin{thm}\label{LIN}
Let us assume that the contact Hamiltonian system \eqref{eq2h} has symmetries $s_0=h,\dots,s_p$
such that $s_0,\dots,s_r$ Jacobi-commute with all symmetries
\[
[s_i,s_j]=0, \quad i=0,\dots,r, \quad j=0,\dots,p, \quad p+r=2n,
\]
 and assume that the contact vector fields $X_0=X_{s_0},\dots,X_r=X_{s_r}$ are complete on $M$.
 Let
 \[
\pi\colon M\setminus M_\mathbf 0 \longrightarrow \mathbb{RP}^p
\]
be the associated momentum map and let $M_{reg}$ be an open subset for which the rank of the differential $d\pi$ is equal to $p$. Then

(i) A commutative Lie algebra of the contact symmetries
\begin{equation}\label{CS}
\mathcal X=\Span\{X_0,\dots,X_r\}
\end{equation}
is tangent to the fibers of the associated momentum map $\pi$ and
\begin{equation}\label{XR}
\dim\mathcal F_x=r+1, \qquad x\in M_{reg},\qquad \mathcal F_x=\Span\{X_0(x),\dots,X_r(x)\}.
\end{equation}

A connected component $M_{\mathbf c}^0$ of the invariant level set
\[
M_{\mathbf c}=\pi^{-1}(\mathbf c)\cap M_{reg}
\]
is diffeomorphic to $\mathbb T^{l} \times \R^{r+1-l}$, for some $l$, $0\le l \le r+1$.
There exist coordinates $\varphi_0,\dots,\varphi_{l-1},x_1,\dots,x_{r+1-l}$
of $\mathbb T^l \times \R^{r+1-l}$, which linearize the equation \eqref{eq2}:
\begin{eqnarray*}
&&\dot\varphi_i=\omega_i=const, \qquad i=0,\dots,l-1,\\
&&\dot x_j=a_j=const, \qquad j=1,\dots,r+1-l.
\end{eqnarray*}

(ii) The contact symmetries \eqref{CS} are also tangent to the zero locus $M_{\mathbf 0}$. Let $M_{\mathbf 0,reg}$ be an open subset of $M_{\mathbf 0}$ such that each point has a neighborhood $U$
with local sections $s_{0U},\dots,s_{pU}$ that are independent in a chart $(U,\alpha_U)$:
\[
M_{\mathbf 0,reg}\cap U=\{x\in U\,\vert\, s_{0U}(x)=0,\dots,s_{pU}(x)=0, \,ds_{0U} \wedge \dots \wedge ds_{pU}\vert_x \ne 0 \}.
\]
Then
\begin{equation}\label{XR*}
\dim\mathcal F_x=r, \qquad x\in M_{\mathbf 0,reg}
\end{equation}
and a connected component $M_{\mathbf 0}^0$ of $M_{\mathbf 0,reg}$
is diffeomorphic to $\mathbb T^{l} \times \R^{r-l}$ with linearized dynamics ($0\le l \le r$).
\end{thm}

If invariant manifolds are compact, then obviously $X_0,\dots,X_r$ are complete.
Also, for a Hamiltonian $h$ one can take any linear combination of commuting symmetries. In Examples \ref{Primer} and \ref{Primer2}
we provide contact integrable systems on
the projective space bundle of the cotangent bundle of the torus $\mathbb T^{n+1}$
 with compact and non-compact invariant manifolds.

We note that  Theorem \ref{LIN} allows us to prove the integrability of contact systems related to contact reductions (see \cite{ZZ}) and contact dual pairs (see \cite{BSTV}) (for Hamiltonian systems see \cite{BJ, Zu, Jo2008}). We will deal with these problems in a separate paper.

\begin{proof}[Proof of Theorem \ref{LIN}]
(i) By taking the sections $s_0,\dots, s_r$ for the Hamiltonians,  from item (ii) of Theorem \ref{pomocna} we obtain that the domains $M_{reg,i}$ ($i=0,\dots,p$)
and the fibers of the mapping $\pi$ are invariant under the action of a commutative Lie algebra of contact symmetries \eqref{CS}.

Furthermore, the vector fields $X_k$, $k=0,\dots,r$ are independent and the fibers of $\pi$
are spanned by $\mathcal X$. Let us assume $r+1\le i \le p$. Since $\omega_{s_i}\ne 0$ on $M_{reg,i}$ (see \eqref{omegaI} and the proof of  Theorem \ref{pomocna}), we have
\begin{equation*}
d(s_0/s_i)\wedge\dots\wedge d(s_{r}/s_{i})\ne 0.
\end{equation*}
Furthermore, we obtain from $[s_k,s_i]=0$ that $s_k/s_i$ are integrals of the Reeb flow on $(M_{reg,i},\alpha_{s_i})$:
\[
Z_{\alpha_{s_i}}(s_k/s_i)=d(s_k/s_i)(Z_{\alpha_{s_i}})=0, \qquad k=0,\dots, r.
\]
Thus, the vector fields $X_0=\Phi^{-1}_{\alpha_i}({s_0/s_i}),\dots,X_r=\Phi^{-1}_{\alpha_i}(X_{s_r/s_i})$ are independent on $M_{reg,i}$ (see Lemma \ref{LEMA} below).

Similarly, for the case $0\le i \le r$, it holds that
\begin{align*}
&\omega_{s_i}=d(s_0/s_i)\wedge\dots\wedge d(s_{i-1}/s_{i})\wedge d(s_{i+1}/s_{i})\wedge\dots\wedge d(s_{r}/s_{i})\ne 0,\\
& Z_{\alpha_{s_i}}(s_k/s_i)=d(s_k/s_i)(Z_{\alpha_{s_i}})=0, \qquad k=0,\dots,i-1,i+1,\dots,r,
\end{align*}
and the vector fields
\begin{align*}
& X_0=\Phi^{-1}_{\alpha_i}({s_0/s_i}),\dots,X_{i-1}=\Phi^{-1}_{\alpha_i}({s_{i-1}/s_i}), X_i=Z_{\alpha_i},\\
& X_{i+1}=\Phi^{-1}_{\alpha_i}({s_{i+1}/s_i}), \dots, X_r=\Phi^{-1}_{\alpha_i}({s_r/s_i})
\end{align*}
are independent on $M_{reg,i}$ (Lemma \ref{LEMA}).

The proof of item (i) now follows from the construction of angle coordinates and the linearization in the Arnold--Liouville theorem \cite{Ar}.

\medskip

(ii) Let $M_{\mathbf 0}^0$ be a connected component of $M_{\mathbf 0,reg}$.
Consider a neighborhood $U$ of $x\in M^0_{\mathbf 0}$
with the local sections $s_{0U},\dots,s_{pU}$.
As in item (i) of  Theorem \ref{pomocna}, we have
\[
0=[s_{kU},s_{iU}]_{\alpha_U}(x) = X_{s_k}(s_{iU})\vert_x-s_{iU}Z_{\alpha_U}(s_{kU})\vert_x=X_{s_k}(s_{iU})\vert_x, \qquad x\in M_{\mathbf 0}^0\cap U,
\]
$k=0,\dots,r$, $i=0,\dots,p$. Thus $\mathcal F_x\subset T_x M_{\mathbf 0}^0$.
On the other hand, we have $ds_{0U}\wedge \dots \wedge ds_{pU}\ne 0$ on $U$, and by Lemma \ref{LEMA}, we get
\[
\dim\mathcal F_x=r \qquad \text{or} \qquad \dim\mathcal F_x=r+1.
\]
Since $\dim M_0^0=r$, we get $\dim\mathcal F_x=r$ and $\mathcal F_x=T_x M_{\mathbf 0}^0$. Thus, $M_\mathbf 0^0$ is diffeomorphic to $\mathbb T^{l} \times \R^{r-l}$ with linearized dynamics.
\end{proof}

\subsection{From integrals to contact symmetries}\label{INT->CS}

In the Subsection \ref{CS->INT} we described how the symmetries of the contact system \eqref{eq2h}
lead to its first integrals. Conversely, we have the following statement

\begin{prop}\label{Stav}
Let $s\in\Gamma(L)$ be a symmetry and let $f\in\mathcal A(M)$ be an integral of the system \eqref{eq2h}.
Then the sections $fh$ and $fs$ of the contact bundle $L$ are also symmetries of the system.
\end{prop}

\begin{proof}
Consider the local contact local chart $(U,\alpha)$ in which $h$ is represented by $f_1$ and $s$ by $f_2$. Then
\[
X_{f_1}(f)=0 \qquad \text{and} \qquad [f_{1},f_{i}]_\alpha = X_{f_{1}}(f_{i})-f_{i}Z_\alpha (f_{1})=0, \qquad i=1,2.
\]

From this follows,
\[
[f_{1},f f_{i}]_\alpha = X_{f_{1}}(f f f_{i})-f f_{i}Z_\alpha (f_{1})=f\big(X_{f_{1}}(f_{i})-f_{i}Z_\alpha (f_{1})\big)=0, \qquad i=1,2.
\]
\end{proof}

Therefore, it is sufficient to consider a formulation of integrability that includes only contact symmetries, and Theorem \ref{LIN} provides unified geometrical setting that relates the contact integrability given in Definitions \ref{defZung} and \ref{defJov}.

\subsection{Technical lemma}

\begin{lem}\label{LEMA}
Let $(M,\alpha)$ be a cooriented contact manifold and let $f_1,\dots,f_r\in\mathcal A(M)$ be arbitrary smooth functions.

(i) The vector fields
$Z_{\alpha},X_{f_1},\dots,X_{f_r}$ are independent if and only if
\begin{equation}\label{lem1}
\alpha\wedge df_1\wedge df_2\wedge \dots\wedge df_r\ne 0.
\end{equation}

(ii) If $df_1\wedge df_2\wedge \dots\wedge df_r\ne 0\vert_x$
then $\dim\Span\{X_{f_1}(x),\dots,X_{f_r}(x)\}$ is either $r$ or $r-1$, while
$\dim\Span\{Z_\alpha(x),X_{f_1}(x),\dots,X_{f_r}(x)\}$ is either $r+1$ or $r$.
 In particular, if $f_1,\dots,f_r$
are integrals of the Reeb flow, then
\[
\dim\Span\{X_{f_1}(x),\dots,X_{f_r}(x)\}=r, \quad \dim\Span\{Z_\alpha(x),X_{f_1}(x),\dots,X_{f_r}(x)\}=r+1.
\]
%
%(iii) Let $g$ be a function that functionally depends on $f_1,\dots,f_r$. Then
%\[
%X_g(x)\in \Span\{Z_\alpha(x),X_{f_1}(x),\dots,X_{f_r}(x)\}.
%\]
\end{lem}

\begin{proof}
(i) Let us assume that there exist $a_0,\dots, a_r\in\R$, $a_0^2+a_1^2+\dots+a_r^2\ne 0$ such that
\[
a_0Z_{\alpha}(x)+a_1 X_{f_1}(x)+\dots+a_r X_{f_r}(x) = 0.
\]
Then
\[
(a_0+a_1+\dots+a_r)Z_{\alpha_r}(x)+\alpha^\sharp(a_1\hat{df_1}(x)+\dots+a_r \hat{df_r}(x)) =0, \quad \hat{df_i}=df_i-df_i(Z_\alpha)\alpha,
\]
i.e. $a_0+a_1+\dots+a_r = 0$ and, since the linear mapping $\alpha^\sharp\vert_x: \mathcal Z^0_x\to \mathcal H_x$ is a bijection,
\begin{equation}\label{lem2}
a_1\hat{df_1}(x)+\dots+a_r \hat{df_r}(x)=0 \quad \Longleftrightarrow \quad
\hat{df_1}\wedge \hat{df_2}\wedge \dots\wedge \hat{df_r}\vert_x = 0.
\end{equation}

Therefore, the contact vector fields $Z_{\alpha},X_{f_1},\dots,X_{f_r}$ are independent if and only if
\eqref{lem1} holds.

\medskip

(ii) Let $f_1,\dots,f_r$ be independent at $x$,
\[
df_1\wedge df_2\wedge \dots\wedge df_r\ne 0\vert_x,
\]
and let
\[
\dim K_x=l, \quad \text{where} \quad K_x=\Span\{\hat{df_1}(x),\dots,\hat{df_r}(x)\}.
\]
Then $l=r$ or $l=r-1$.

Let
\[
\hat{\mathcal K}_x=\alpha^\sharp\vert_x(K_x)\subset\mathcal H_x,  \qquad \mathcal K_x=\Span\{X_{f_1}(x),\dots,X_{f_r}(x)\}.
\]
Then $\dim\mathcal K_x\le r$ and
since $\alpha^\sharp\vert_x$ is an isomorphism between $\mathcal Z^0_x$ and $\mathcal Z_x$,
$\dim\hat{\mathcal K}_x=l$.

If $Z_\alpha(x)\in \mathcal K_x$, then
$\dim \mathcal K_x=\dim \hat{\mathcal K}_x+1=l+1$, and this is only possible if $l=r-1$. Therefore, in this case we have
\[
\dim\Span\{Z_{\alpha}(x),X_{f_1}(x),\dots,X_{f_r}(x)\}=\dim\mathcal K_x=r.
\]
Note also that in this case there exist $a_1,\dots,a_r\in\R$, so that
\[
a_1 \hat{df}_1(x)+\dots+a_r \hat{df}_r(x)=0, \qquad a_1 f_1(x)+\dots+a_r f(x)\ne 0.
\]

If $Z_\alpha(x) \notin \mathcal K_x$, then $\dim\mathcal K_x=\dim\hat{\mathcal K}_x=l$.
There are two possibilities: $l=r-1$ or $l=r$.
If $l=r-1$, then there are $a_1,\dots,a_r\in\R$, $a_1^2+\dots+a_r^2\ne 0$, so that
\[
a_1 \hat{df}_1(x)+\dots+a_r \hat{df}_r(x)=0, \qquad a_1 f_1(x)+\dots+a_r f(x)= 0
\]
and
\[
\dim\Span\{Z_{\alpha}(x),X_{f_1}(x),\dots,X_{f_r}(x)\}=\dim\mathcal K_x+1=l+1=r.
\]

If $l=r$, then
\begin{equation}\label{nezavisnost}
\dim\Span\{Z_{\alpha}(x),X_{f_1}(x),\dots,X_{f_r}(x)\}=r+1.
\end{equation}

In particular, if the functions $f_1,\dots,f_r$ are the first integrals of the Reeb flow ($df_i=\hat{df_i}$, $i=1,\dots,r$), then $l=r$, $\dim\mathcal K_x=\dim\hat{\mathcal K}_x=r$, and \eqref{nezavisnost} holds.
%(iii) The statement follows directly from item (i). The dependency can be given explicitly.
%Let $g=G(f_1,\dots,f_r)$. Then:
%\begin{align}
% X_g=&g Z_{\alpha}+\alpha_i^\sharp\big(dg-Z_{\alpha}(g)\alpha\big) \nonumber \\
%=& g Z_{\alpha}+\alpha^\sharp\Big( \sum_{k=1}^r \frac{\partial G}{\partial f_k}df_k-\sum_{k=1}^r \frac{\partial G}{\partial f_k}Z_{\alpha_i}(f_k)\alpha\Big)\label{Gf} \\
%=&\big(g- \sum_{k=1}^r f_k \frac{\partial G}{\partial f_k}\big)Z_{\alpha}+\sum_{k=1}^r \frac{\partial G}{\partial f_k} X_{f_k} \nonumber
%\qedhere
%\end{align}
\end{proof}

\section{Foliations and contact normal forms}\label{sec4}

\subsection{Foliations}

As above, we denote by $\mathcal F$ the foliation determined by commutative contact symmetries $\mathcal X$, and by $\mathcal E$ the distribution determined by contact symmetries $\mathcal Y$.

The invariant set
\[
\Sigma=(s_0)\cap \dots\cap (s_r) \cap M_{reg}
\]
is either an empty set or a $p$--dimensional submanifold of $M_{reg}$. In the latter case, the inequality $r < n < p$ must be satisfied.
We note that the dynamics on $\Sigma$ is also linearized. If the level sets of $\pi$ on $M_{reg}$ are compact, $\Sigma$ is folded on invariant $(r+1)$-dimensional tori.

\begin{thm}\label{foliacije}
Let us assume that the contact system \eqref{eq2h} is integrable in the sense of Theorem \ref{LIN}. Then:

(i) $\mathcal F$ is $\alpha_{s_i}$-complete pre-isotropic foliation on $(M_{reg,i},\alpha_{s_i})$, for $i=0,\dots,r$.
In the case that the momentum mapping $\pi$ is proper, $(M_{reg}\setminus\Sigma,\mathcal X,\mathcal H)$
 is a complete pre-isotropic contact structure.

(ii) The set $\Sigma$ can be defined as a set where the fibers of $\pi$ are not transversal to the contact distribution $\mathcal H$:
\begin{equation}\label{sig1}
\Sigma=\{x\in M_{reg} \, \vert\, \mathcal F_x \subset \mathcal H_x\}
\end{equation}

(iii)
The distribution $\mathcal E$ is a regular $p+1$-dimensional foliation on $M_{reg}\setminus\Sigma$ and
$\Sigma$ is a singular leaf of $\mathcal E$: $T_x\Sigma=\mathcal E_x$,
\begin{equation}\label{sig2}
\Sigma=\{x\in M_{reg} \, \vert\, \dim\mathcal E_x = p\}.
\end{equation}
\end{thm}

\begin{proof}
(i) Let $x\in M_{reg}\setminus \Sigma$. Then $s_i(x)\ne 0$ and $x\in M_{reg,i}$ for any index $i$, $0\le i\le r$.
The Reeb vector field on $(M_{reg,i},\alpha_i)$ is $X_{s_i}$ and the mapping \eqref{PIi} satisfies the conditions of Theorem \ref{teo1}: $\mathcal F$ is $\alpha_{s_i}$-complete pre-isotropic foliation on $(M_{reg,i},\alpha_{s_i})$ (for this property we do not need the map \eqref{PIi} to have compact fibers).

\medskip

(ii) The identity \eqref{sig1} is clear -- the contact vector field $X_s$ is tangent to $\mathcal H$ at $x$ if and only if $s(x)=0$.

\medskip

(iii) It follows from (i) that for $i=0,\dots,r$, $\mathcal E\vert_{M_{reg,i}}$ is a foliation (see Proposition \ref{STAV}). Also
\begin{equation}\label{dex}
\dim\mathcal E_x=p+1, \qquad x\in M_{reg}\setminus \Sigma.
\end{equation}

On the other hand, since the contact vector fields $X_{s_0}, \dots, X_{s_p}$ are tangent to $\Sigma$ (set $s=s_i$, $h=s_j$, $i=0,\dots,r$, $j=0,\dots,p$ in Lemma \ref{pomocnaLema}),
due to \eqref{dex} we can define $\Sigma$ by \eqref{sig2}.
In particular, $\Sigma$ is a singular leaf of foliation $\mathcal E$: $T_x\Sigma=\mathcal E_x$.
\end{proof}

In the case that $r=p=n$, then $\Sigma=\emptyset$ and we have \emph{commutative contact integrability}: $\mathcal F=\mathcal E$ is a pre-Legendrian foliation on $M_{reg}$.
If the momentum mapping $\pi$ is proper, then $(M_{reg},\mathcal X,\mathcal H)$ is a complete pre-Legendrian contact structure.

Theorems \ref{LIN} and \ref{foliacije} provide a global framework that explains how transversal and non-transversal cases are related.
With the symmetries $\mathbf Y$ satisfying the conditions of  Theorem \ref{LIN}, we associate foliations $\mathcal F$ and $\mathcal E$,
\[
\mathcal F\subset \mathcal E,
\]
generated by commuting symmetries and by all symmetries, respectively. Then we have the decomposition of $M$
in terms of foliations described in Table \ref{TAB}.
\begin{table}[h]
\begin{tabular}{|c||c|c|c|c|}
\hline $M_{reg}\setminus \Sigma$ & $\mathcal F_x \,\, \text{is transverse to}\,\, \mathcal H_x$ &  $\dim \mathcal E_x=p+1$  & $\dim \mathcal F_x= r+1$ &  \\
\hline $\Sigma$ & $\mathcal F_x \subset \mathcal H_x$ &  $\dim \mathcal E_x=p$  & $\dim \mathcal F_x= r+1$ & $T_x\Sigma =\mathcal E_x$ \\
\hline  $M_{\mathbf 0, reg}$ & $\mathcal F_x\subset \mathcal E_x \subset \mathcal H_x$ &  $\dim \mathcal E_x=p$  & $\dim \mathcal F_x= r$ & $T_xM_{\mathbf 0,reg} =\mathcal F_x$  \\
\hline
\end{tabular}\vspace{0.5em}
\caption{The decomposition of $M$ in terms of foliations $\mathcal F\subset \mathcal E$.}\label{TAB}
\end{table}

\begin{cor}
Let us assume that the contact system \eqref{eq2h} is integrable in the sense of Theorem \ref{LIN} with compact regular invariant manifolds. Then the contact manifold $M$ is almost everywhere foliated on invariant tori $T$, which are
\begin{itemize}
\item[(i)] pre-isotropic and of dimension $r+1$ for $T\subset M_{reg}\setminus \Sigma$;
\item [(ii)] isotropic and of dimension $r+1$ for $T\subset \Sigma$ ($\dim\Sigma=p=2n-r$);
\item[(iii)] isotropic and of dimension $r$ for $T\subset M_{\mathbf 0, reg}$ ($\dim M_{\mathbf 0, reg}=r$).
\end{itemize}
\end{cor}

\subsection{Contact normal forms}

Let $M_{\mathbf c}^0$ be an invariant $(r+1)$-dimensional
pre-isotropic  torus which is contained in $M_{reg}\setminus \Sigma$, i.e
 $M_{\mathbf c}^0\subset M_{reg,i}$ for an index $i$, $0\le i\le r$.
Then we can apply Theorem \ref{teo1}, with the integrals $f_1,\dots,f_p$ are given by
\[
(f_1,\dots,f_p)=(\phi_{s_i}(s_0),\dots,\phi_{s_i}(s_{i-1}),\phi_{s_i}(s_{i+1}),\dots,\phi_{s_i}(s_p)).
\]

Thus, there is a $\pi$-invariant neighborhood $U$ of $M_{\mathbf c}^0$ with coordinates
\[
(\varphi_0,\dots,\varphi_r,I_1,\dots,I_r,q_1,\dots,q_s,p_1,\dots,p_s)
\]
where the contact form $\alpha_{s_i}$ has the form \eqref{canonical*}.

Consider the case where $ \Sigma \ne \emptyset $ and an invariant
$(r+1)$-dimensional isotropic torus $M_{\mathbf c}^0$ which is a subset of $\Sigma$ ($r\le n-1$). There is an index
$i$, $r+1\le i \le p$, so that $M_{\mathbf c}^0$ is the subset of $M_{reg,i}$.
For example, we can set $i=p$.
It follows that there are $\pi$-invariant neighborhoods
 $U\subset M_{reg,p}$ of $M_{\mathbf c}^0$ such that
\[
\pi_p=\vartheta_p\circ\pi=(\phi_{s_p}(s_0),\dots,\phi_{s_p}(s_{p-1}))\colon U\to \R^p
\]
is a toric fibration with fibers generated by the abelian Lie algebra of contact symmetries $\mathcal X=\Span\{X_{s_0},\dots,X_{s_r}\}$. Then, as follows from \cite{Zung}, subsection 4.4, by possibly shrinking of $U$, there exist coordinates
\[
\phi=(\theta_0,\dots,\theta_r,z_0,\dots,z_r,x_1,\dots,x_s,y_1,\dots,y_{s-1}): U \longrightarrow \T^{r+1}\times D
\]
and a local contact form $\alpha_U$ defining $\mathcal H\vert_U$ so that the integrals depend on $x,y,z$ and
\begin{equation}\label{zung2}
\alpha_U=z_0d\theta_0+z_1d\theta_1+\dots+z_rd\theta_r+y_1dx_1+\dots+y_{s-1}dx_{s-1}+dx_s.
\end{equation}
Here, the coordinates $z_0,\dots,z_r$ are equal to zero at $\Sigma\cap U$.

To illustrate Theorems \ref{LIN} and \ref{foliacije} we give the following simple example.

\begin{exm}\label{Primer}
It is known that the projective space bundle of the cotangent bundle carries the natural contact structure.
The contact structure over the torus $\mathbb T^{n+1}$
can be described as follows.
Consider the projectivization of the cotangent bundle 
\[
M=\mathbb T^{n+1}(\varphi_0,\dots,\varphi_n)\times\mathbb{RP}^{n}([y_0,\dots,y_{n}]),
\]
 where
$[y_0,\dots,y_{n}]$ are homogeneous coordinates on the projective space $\mathbb{RP}^{n}$. The contact structure in toroidal domains
\begin{align*}
&\psi_i\colon V_i=M\setminus \{y_i=0\}\longrightarrow \mathbb T^{n+1}\times \mathbb R^{n}, \\
&\psi_i(\varphi_0,\dots,\varphi_n,[y_0,\dots,y_{n}])=(\varphi_0,\dots,\varphi_n,J_0^i,\dots,J_{i-1}^i,J_{i+1}^i,\dots,J_{n}^i),\quad J^i_j=\frac{y_{j}}{y_i}
\end{align*}
is defined by local contact forms
\begin{align*}
\alpha_{V_i}=\sum_{j=0}^{n} J^i_j d\varphi_j, \qquad J^i_i\equiv 1, \qquad i=0,\dots,n.
\end{align*}

The transition functions of the contact line bundle $L$  are
$s_{V_iV_j}={y_j}/{y_i}=J^i_j$ (the pull back of the tautological line bundle over $\mathbb{RP}^n$).

Consider the standards sections $s_0,\dots,s_{n}$ of $L$, which are locally defined by
$
s_j\vert_{V_i}=J^i_j
$
and the contact flow $X_h$ determined by the section
$h=\omega_0 s_0+\dots+\omega_{n-1} s_{n-1}$, $\omega_j=const$. The contact system in the chart $V_i$ is
\[
\dot\varphi_0=\omega_0,\, \dots, \, \dot\varphi_{n-1}=\omega_{n-1}, \, \dot \varphi_n=0, \, \dot J^i_0=0,\,\dots,\dot J^{i}_{i-1}=0,\,\dot J^{i}_{i+1}=0, \,
\dots, \dot J^{i}_{n}=0.
\]

According to Proposition \ref{Stav}, the system has symmetries
$\{s_0,\dots,s_{n}, f_0 s_0,\dots, f_n s_{n}\}$,
where $f_i$ are arbitrary smooth functions on $M$ that depend only on the angle variable $\varphi_n$.
However, not all of them induce independent integrals.
Therefore, we fix $k$, $0\le k\le n$, and consider the symmetries
\[
\mathbf Y=\Span \{s_0,s_1,\dots,s_{n}, f s_k\},
\]
together with the associated momentum map
\[
\pi(x)=[s_0(x),s_1(x),\dots,s_{n}(x), f(x) s_k(x)], \qquad x\in M,
\]
and commuting symmetries $\mathbf X=\Span\{s_0,s_1,\dots,s_{n-1}\}\subset \mathbf Y$.
We assume that $f$ is analytic, only depends on $\varphi_n$, and is always greater than zero.

Then the conditions of Theorem \ref{LIN} with $r+1=n$, $p=n+1$ are fulfilled.
We have
\begin{align*}
& (s_i)=M\setminus V_i,\qquad i=0,\dots,n,\\
& (f s_{k})=(s_k)=M\setminus V_k.
\end{align*}

Thus the set $M_\mathbf 0$ is empty and the $(n+1)$-dimensional manifold $\Sigma=(s_0)\cap \dots \cap (s_{n-1})$ is a subset of $V_{n}$. On $V_{n}$ the momentum mapping has the form
\[
\pi\vert_{V_{n}}=[J^{n}_0,J^{n}_1,\dots,J^{n}_{n-1}, 1, f(\varphi_n) J^{n}_k],
\]
the contact form
$\alpha_{V_{n}}=J^{n}_0 d\varphi_0+\dots+ J^n_{n-1} d\varphi_{n-1}+d\varphi_n$ is already in the normal form \eqref{zung2}, and
\[
\Sigma=\{x\in V_{n}\,\vert\, J^{n}_0=\dots=J^{n}_{n-1}=0\}.
\]

Outside of $\Sigma$ we have integrability in the sense of Definition \ref{defJov}. For example, consider the domain $V_0$.
Then
\[
\pi\vert_{V_{0}}=[1,J^{0}_1,\dots,J^{0}_{n}, f(\varphi_n) J^{0}_k]
\]
and
$\alpha_{V_{0}}=d\varphi_0+J^{0}_1 d\varphi_1+\dots+J^0_n d\varphi_n$ has the contact normal form \eqref{canonical*}. From the momentum map we have integrals $\{J^{0}_1,\dots,J^{0}_{n}, f(\varphi_n)\}$
define the invariant $n$-dimensional pre-isotropic tori.
\end{exm}

\section{Commutative integrability on cooriented contact manifolds}\label{sec5}

We can formulate  Theorems  \ref{LIN} and \ref{foliacije} for cooriented contact manifolds $(M,\alpha)$, with the symmetries $s_i\in\Gamma(L)$ are replaced by the functions $f_i\in\mathcal A(M)$.
For the sake of simplicity, we consider the commutative case in which $\Sigma=\emptyset$.

Let us consider the contact Hamiltonian equation \eqref{eq_f}
with the Hamiltonian $h\in\mathcal A(M)$. We assume that we have $n+1$ contact symmetries $f_0=h,f_1,\dots,f_n\in\mathcal A(M)$ that are in involution:
\begin{equation}\label{involucija-c}
[f_i,f_j]_\alpha=0, \qquad i,j=0,\dots,n.
\end{equation}
Let
 \begin{equation}\label{NewMM}
\pi: M\setminus M_\mathbf 0 \longrightarrow \mathbb{RP}^n, \quad \pi(x)=[f_0(x),\dots,f_n(x)]
\end{equation}
is the associated momentum mapping, where $M_\mathbf 0$ is the zero level set:
\begin{equation}\label{ZERO}
M_\mathbf 0=\{x\in M\,\vert\, f_0(x)=0, f_1(x)=0,\dots, f_n(x)=0\}.
\end{equation}

As above, let $M_{reg}$ be an open subset of $M\setminus M_{\mathbf 0}$ in which the rank of the differential $d\pi(x)$ is equal to $n$
and let $M_{reg,i}$ be the invariant set in which the function $f_i$ is different from zero:
\[
M_{reg,i}= \{x\in M_{reg}\,\vert\, f_i(x)\ne 0\}\, \qquad i=0,\dots,n.
\]

Then the foliations $\mathcal F$ and $\mathcal E$ coincide
$
\mathcal F=\mathcal E,
$
and the symmetries are transversal to the contact structure $\mathcal H$ at $M_{reg}$ ($\Sigma=\emptyset$).
As already mentioned, if $\pi$ is proper, $(M_{reg},\mathcal H,\mathcal X)$ is a complete pre-Legendrian contact structure and the system is contact commutatively integrable in the sense of Definition \ref{defJov} (see also \cite{BM}).
In particular, in invariant domains $M_{reg,i}$, $i=0,\dots,n$ we have the situation as in Theorem \ref{teo1}.

We can summarize the commutative case in the next statement.

\begin{thm}\label{komutativna}
Let us consider the contact Hamiltonian equation \eqref{eq_f}
with the Hamiltonian $h\in\mathcal A(M)$ and assume that we have $n+1$ contact symmetries $f_0=h,f_1,\dots,f_n\in\mathcal A(M)$ which are in involution with respect to the Jacobi bracket \eqref{involucija-c} on $(M,\alpha)$ so that the associated fibration \eqref{NewMM} is proper.

(i) A commutative Lie algebra of contact symmetries $\mathcal X$ spanned by
$X_{f_0},\dots,X_{f_n}$ is tangent to the fibers of the momentum mapping \eqref{NewMM} and $\mathcal X$ is transversal to the contact structure $\mathcal H$ at $M_{reg}$:
$(M_{reg},\mathcal H,\mathcal X)$ a complete pre-Legendrian contact structure.

(ii) Let $M_{\mathbf c}^0$ be a regular connected component of the invariant level set
$M_{\mathbf c}=\pi^{-1}(\mathbf c)\cap M_{reg}$. It is diffeomorphic to a torus $\mathbb T^{n+1}$.
Let us assume that $M_{\mathbf c}^0$ belongs to the invariant set $M_{reg,i}$ ($c_i\ne 0$).
Then there is an open $\pi$-invariant neighborhood $U$ of $M_{\mathbf c}^0$ with Liouville coordinates
\begin{equation*}
(\varphi,I)=(\varphi_0,\varphi_1,\dots,\varphi_n,I_1,\dots,I_n),
\end{equation*}
so that the contact form $\alpha_{f_i}=\frac{1}{f_i}\alpha$ has the following canonical form
\begin{equation*}
\alpha_i=I_0d\varphi_0+I_1d\varphi_1+\dots+I_n d\varphi_n,
\end{equation*}
where the action integrals $I_k$ have the form
\[
I_k=I_k(f_0/f_i,\dots,f_{i-1}/f_i,f_{i+1}/f_i,\dots,f_n/f_i), \qquad k=0,\dots,n.
\]

The flow of $X_h$ on invariant tori is quasi-periodic
\begin{equation*}\label{komutativnoNamotavanje}
(\varphi_0,\varphi_1,\dots,\varphi_n) \longmapsto
(\varphi_0+t\omega_0,\varphi_1+t\omega_1,\dots,\varphi_n+t\omega_n).
\quad t\in\R,
\end{equation*}
The Hamiltonian $h$ and the frequencies $\omega_0,\dots,\omega_n$ depend only on the action variables $I_1,\dots,I_n$.

(iii) The contact symmetries $\mathcal X$ are also tangent to the zero level set \eqref{ZERO}.
Regular connected components of
$M_{\mathbf 0}$ are invariant Legendrian $n$-dimensional tori with linearized dynamics.
\end{thm}

The above formulation allows us to work with dissipative Hamiltonian systems where the Hamiltonian does not have to be preserved by the original Reeb vector field \cite{deLeon}.
On the other hand, note that the contact Hamiltonian vector field $X_h$ in the domain $M_{reg,0}$ is the Reeb vector field with respect to $\alpha_h=\frac{1}{h}\alpha$, while it commutates with the Reeb vector field $Z_{\alpha_{f_i}}$ in the domains $M_{reg,i}$ with respect to the contact forms $\alpha_{f_i}=\frac{1}{f_i}\alpha$, $i=1,\dots,n$.

Recently, the statement analogous to item (ii) of  Theorem \ref{komutativna} was given in \cite{Col}, with the proof based on the symplectization of contact manifolds and the Liouville-Arnold theorem for exact symplectic manifolds for commuting integrals which are homogeneous of degree 1.

\begin{exm}\label{Primer2}
We use the notation of Example \ref{Primer}. Here we also assume that $f$ is analytic, depends only on $\varphi_n$, but now with possible zeros. We consider Jacobi commuting sections
$\mathbf X=\Span\{s_0,\dots,s_{n-1},f s_{n}\}$ and
the Hamiltonian $h=\omega_0 s_0+\dots+\omega_{n-1} s_n+ f s_{n}\in \mathbf X$, $\omega_j=const$.
The zero-locus  $M_{\mathbf 0}$ is given by
\[
M_{\mathbf 0}=(s_0)\cap (s_1) \dots \cap (s_{n-1}) \cap (f),
\]
where $(f)=\mathbb T^{n}\{\varphi_0,\dots,\varphi_{n-1}\}\times \mathbb{RP}^n([y_0,\dots,y_n])\times \Xi$, $\Xi=\{\varphi_n\,\vert\, f(\varphi_n)=0\}$.
Thus
\[
M_{\mathbf 0}=\mathbb T^{n}\{\varphi_0,\dots,\varphi_{n-1}\}\times \{[0,\dots,0,y_n]\vert y_n\ne 0\}  \times \Xi \cong \mathbb{T}^n\times \Xi.
\]

We have the momentum mapping $\pi\colon M\setminus M_{\mathbf 0} \longrightarrow \mathbb{RP}^{n}$,
\[
\pi(x)=[s_0(x),\dots,s_{n-1}(x),f(x) s_{n}(x)].
\]

Since $M$ is compact, all contact Hamiltonian vector fields are complete and
the contact flow of $X_h$ is completely integrable in the sense of Theorem \ref{LIN}:
$M$ is foliated on $(n+1)$-dimensional
pre-Legendrian invariant manifolds within $M\setminus M_{\mathbf 0}$ and Legendrian invaiant manifolds within $M_{\mathbf 0}$.

The topology of invariant manifolds depends on the analytic function $f$.

Since $M_{\mathbf 0}$ is a subset of the domain $V_n$,
we consider the system within $V_{n}$ and denote
$p_j=J^{n}_j$, $j=0,\dots,n-1$.
Then the contact form reads
\[
\alpha=\alpha_{V_{n}}=p_0d\varphi_0+\dots+p_{n-1} d\varphi_{n-1}+d\varphi_n.
\]

Note that $(V_n,\alpha)$ can be considered as the product of cotangent bundle of a $n$-dimensional torus and a circle:
\[
N=T^*\mathbb T^{n}(p_0,\dots,p_{n-1},\varphi_0,\dots,\varphi_{n-1})\times S^1(\varphi_n)
\]
endowed with the canonical contact structure   (see Example \ref{canonical}).
Thus, for the Hamiltonian function $h=\omega_0 p_0+\dots+\omega_{n-1} p_{n-1}+ f(\varphi_n)$,
we get the following simple dissipative contact Hamiltonian system on $T^*\mathbb T^{n}\times S^1$
\begin{align*}
&\dot \varphi_i=\frac{\partial h}{\partial p_i}=\omega_i,\qquad \dot p_i=-\frac{\partial h}{\partial q_i}+p_i\frac{\partial h}{\partial \varphi_n}= \frac{df(\varphi_n)}{d\varphi_n} p_i, \qquad  i=0,\dots,n-1.\\
& \dot \varphi_n=h-\sum_{i=0}^{n-1} p_i\frac{\partial h}{\partial p_i}=f(\varphi_n),
\end{align*}
that fulfils Theorem \ref{komutativna} with the momentum mapping
\[
\pi=[p_0,p_1,\dots,p_{n-1},f(\varphi_n)]
\]
 defined in $N\setminus N_{\mathbf 0}$,
\[
N_{\mathbf 0}=\{p_0=0, \, p_1=0,\, \dots, \, p_{n-1}=0, \, f(\varphi_n)=0\}\cong \mathbb{T}^n\times \Xi.
\]

For example, in the domain where $f\ne 0$, we get Jacobi commuting integrals
\[
f_0=p_0/f, \, \dots, \, f_{n-1}=p_{n-1}/f,
\]
with respect to the contact form $\frac{1}{f}\alpha$.

If $f>0$, then $\Xi=M_{\mathbf 0}=N_{\mathbf 0}=\emptyset$ and $M$ is foliated on invariant pre-Legendrian  tori.
Within $N\subset M$, the tori are given as the level sets of integrals $f_0,\dots,f_{n-1}$.

On the other side, when $\Xi \ne \emptyset$, the manifold $M_{\mathbf 0}=N_{\mathbf 0}$ is the union of
invariant Legendrian  tori, but pre-Legendrian invariant manifolds within $M\setminus M_{\mathbf 0}$ (or within $N\setminus N_{\mathbf 0}$) are not compact.
\end{exm}

\subsection*{Acknowledgments}
I am grateful to the referee for valuable remarks.
The research was supported by the Project no. 7744592, MEGIC "Integrability
and Extremal Problems in Mechanics, Geometry and Combinatorics" of the Science Fund
of Serbia.


\begin{thebibliography}{84}

\bibitem{Ar}
{\srrm V. I. Arnol\m d,} {\srit Matematicheskie metody
klassichesko{\ji} mehaniki}{\srrm,  3-e izd.,  Nauka, Moskva 1989} (Russian).
English translation:

V. I. Arnol'd, \emph{Mathematical methods of classical
mechanics}, 3rd ed., Springer-Verlag, 1989.

\bibitem{BM} A. Banyaga, P. Molino,  \emph{G\' eom\' etrie des
formes de contact compl\'etement int\'egrables de type torique},
S\'eminare Gaston Darboux, Montpellier (1991-92), 1-25 (French).
English translation:

\emph{Complete Integrability in Contact Geometry}, Appendix B in \emph{A Brief Introduction to Symplectic and Contact Manifolds} by
A. Banyaga  and D. F. Houenou, pp. 107--158,
Nankai Tracts in Mathematics \textbf{15}, World Scientific, 2016.

\bibitem{BSTV}
A. M. Blaga, M. A. Salazar, A. G. Tortorella, C. Vizman, \emph{Contact Dual Pairs}, International Mathematics Research Notices, \textbf{2020} (2020) 8818--8877,
 	arXiv:1903.05250

\bibitem{Bo} O. I. Bogoyavlenskij, \emph{Extended integrability and bi-hamiltonian systems},
Comm. Math. Phys. \textbf{196} (1998), no. 1, 19--51.

\bibitem{BJ} A. V. Bolsinov, B. Jovanovi\' c,
\emph{Non-commutative integrability, moment map and geodesic
flows}. Annals of Global Analysis and Geometry {\bf 23} (2003) no. 4,
305--322, arXiv:math-ph/0109031.

\bibitem{Boyer} C. P. Boyer, \emph{Completely integrable contact
Hamiltonian systems and toric contact structures on $S^2\times
S^3$}, SIGMA  {\bf 7} (2011), Paper 058, 22 pp, arXiv: 1101.5587
[math.SG]


\bibitem{CFG} J. F. Carinena, F. Falceto,  J. Grabowski,
\emph{Solvability of a Lie algebra of vector fields implies their integrability by quadratures},
J. Phys. A: Math. Theor. {\bf 49} (2016) 425202,  	arXiv:1606.02472 [math-ph].

\bibitem{Col}
L. Colombo, M. de León, M. Lainz, A. López-Gordón,
\emph{Liouville-Arnold theorem for contact Hamiltonian systems}, 	arXiv:2302.12061 [math.SG].

\bibitem{FS}
F. Fasso, N. Sansonetto, \emph{Integrable almost-symplectic Hamiltonian systems}, J. Math. Phys. {\textbf 48}
(2007), no. 9, 092902, (13 pp).


\bibitem{Gr} J. W. Grey,  \emph{Some global properties of contact
structures}, Ann. of Math. {\bf 69} (1959) 421--450.



\bibitem{GH}
 P. Griffiths, J. Harris,  \emph{Principles of algebraic geometry}, Wiley, 1978, New York.

\bibitem{deLeon} M. de León, M. L. Valcázar, \emph{Infinitesimal symmetries in contact Hamiltonian systems},
Journal of Geometry and Physics, \textbf{153} (2020), 103651, arXiv:1909.07892 [math-ph].

%\bibitem{Ji} K. Jiang,
%\emph{Local normal forms of smooth weakly hyperbolic integrable systems}, Regular and Chaotic Dynamics, {\bf 21} (2016) 18--23.

\bibitem{Jo2008} B. Jovanovi\' c, \emph{Symmetries and Integrability}, Publications de l'institut math\'ematique {\textbf 84(98)} (2008), 1--36,
arXiv:0812.4398 [math.SG].

\bibitem{Jo} B. Jovanovi\' c, \emph{Noncommutative integrability and action angle variables in contact
geometry}, Journal of Symplectic Geometry,  \textbf{10} (2012), 535--562, arXiv:1103.3611.

%\bibitem{Jo1} B. Jovanovi\' c, \emph{Noether symmetries and integrability in time-dependent mechanics}, Theoretical and Applied Mechanics,  \textbf{43} (2016)

\bibitem{JJ1} B. Jovanovi\'c, V. Jovanovi\'c, \emph{Contact flows and integrable systems}, J. Geom. Phys. {\bf 87} (2015) 217--232, arXiv:1212.2918.

\bibitem{JJ2}
B. Jovanovi\' c, V.  Jovanovi\' c,
\emph{Virtual billiards in pseudo--Euclidean spaces: discrete Hamiltonian and contact integrability},
{Discrete and Continuous Dynamical Systems - Series A},
{\bf 37} (2017), no. 10, 5163–-5190, arXiv:1510.04037.

\bibitem{JJ3}
B. Jovanovi\' c, V.  Jovanovi\' c,
\emph{Heisenberg model in pseudo-Euclidean spaces II}, {Regular and Chaotic Dynamics}, \textbf{23} (2018) 289--308,
arXiv:1808.10783.

\bibitem{JL} B. Jovanovi\' c, K.  Luki\' c,
\emph{Integrable systems in cosymplectic geometry}, J. Phys. A: Math. Theor. \textbf{56} (2023) 015201 (18pp), arXiv:2212.09427.

\bibitem{KT} B. Khesin, S. Tabachnikov, \emph{Contact complete integrability}, Regular and Chaotic Dynamics, \textbf{15} (2010) 504--520, arXiv:0910.0375.

\bibitem{Koz} V. V. Kozlov,
\emph{Symmetries, topology and resonances in Hamiltonian mechanics},
Ergebnisse der Mathematik und ihrer Grenzgebiete (3), 31.
Springer-Verlag, Berlin, 1996.

\bibitem{Koz2} V. V. Kozlov, \emph{The Euler-Jacobi-Lie Integrability Theorem}, Regular and Chaotic Dynamics, {\bf 18} (2013) 329--343.

\bibitem{LMV}
C. Laurent-Gengoux, E. Miranda, P. Vanhaecke, \emph{Action-angle coordinates for integrable
systems on Poisson manifolds}, Int. Math. Res. Not.,  \textbf{2011} (2011) Issue 8, 1839--1869, arXiv: arxiv.0805.1679 [math.SG].

\bibitem{KM} A. Kiesenhofer, E. Miranda,
\emph{Non-commutative integrable systems on b-symplectic manifolds},
Regular and Chaotic Dynamics, \textbf{21} (2016), no. 6, 643--659.

\bibitem{LM} P. Libermann, C. Marle, \emph{Symplectic Geometry and Analytical Mechanics}, Riedel, Dordrecht, 1987.

\bibitem{L}
P. Libermann, \emph{Cartan–Darboux theorems for Pfaffian forms on foliated manifolds},
Proc. VIth Int. Colloq. Dif. Geom. Santiago 1989, 125–144. See also: On symplectic
and contact groupoids, Differential Geometry and Its Application, Proc. Conf. Opava,
1992, 29–45.

\bibitem{MF}
{\srrm A. S. Mishchenko, A. T. Fomenko,} {\srit
Obobshchenny{\ji} metod Liuvill{\ja} integrirovani{\ja}
gamiltonovyh sistem}{\srrm, Funkc. analiz i ego prilozh.
\textbf{12}(2) (1978), 46--56} (Russian); English translation:

A. S. Mishchenko, A. T. Fomenko, \emph{Generalized Liouville
method of integration of Hamiltonian systems}. Funkts. Anal.
Prilozh. {\bf 12}, No.2, 46-56  (1978)  Funct. Anal. Appl. {\bf
12}, (1978) 113--121

\bibitem{N}
{\srrm N. N. Nehoroshev,} {\srit Peremennye de{\ji}stvie--ugol i
ih obobshcheni{\ja}}{\srrm, Tr. Mosk. Mat. O.-va. \textbf{26}
(1972), 181--198} (Russian). English translation:

N. N. Nekhoroshev, \emph{Action-angle variables and their
generalization}, {Trans. Mosc. Math. Soc.} {\bf 26} (1972)  180--198.

\bibitem{SS}
M. A. Salazar, D. Sepe,
\emph{Contact Isotropic Realisations of Jacobi Manifolds via Spencer Operators},
SIGMA Symmetry Integrability Geom. Methods Appl, \textbf{13} (2017) 33--76, .

\bibitem{TF} {\srrm V. V. Trofimov, A. T. Fomenko},
{\srit Algebra i geometriya integriruemykh gamilton\m ovykh differentsial\m nykh uravneni{\ji},}
{\srrm Moskva, Faktorial, 1995},
 V. V. Trofimov, A. T. Fomenko,
\emph{Algebra and geometry of integrable Hamiltonian differential equations.}
Moskva, Faktorial, 1995 (Russian).

\bibitem{ZZ}
M. Zambon and C. Zhu, \emph{Contact reduction and groupoid actions}. Trans. Amer. Math. Soc., \textbf{358}(2006) 1365--1401.

\bibitem{Zu} N. T. Zung,
\emph{Torus actions and integrable systems}, In:  A. V. Bolsinov,
A. T. Fomenko, A. A. Oshemkov (eds.), {Topological Methods in
the Theory of Integrable Systems} 289--328, Cambridge Scientific
Publ., (2006), arXiv: math.DS/0407455

\bibitem{Zung} N. T. Zung, \emph{A Conceptual Approach to the Problem of Action-Angle Variables}, Arch. Ration. Mech. Anal. \textbf{229} (2018) 789--833,
arXiv:1706.08859 [math.DS].

\end{thebibliography}
\end{document}